\documentclass[12pt,reqno]{amsart}
\usepackage{amsmath,amsthm,amsfonts,amssymb,amscd,amstext}
\usepackage[latin1]{inputenc}
\usepackage[dvips]{graphicx}
\usepackage{psfrag}
\usepackage{euscript}
\usepackage{a4wide}
\usepackage{mathrsfs}
\usepackage{xcolor}
\usepackage{mathtools}
\usepackage[colorlinks=true,linkcolor=blue, urlcolor=red, citecolor=blue, backref]{hyperref}

\numberwithin{equation}{section}

\theoremstyle{plain}
\newtheorem{maintheorem}{Theorem}

\newtheorem{theorem}{Theorem}[section]
\newtheorem{proposition}[theorem]{Proposition}
\newtheorem{corollary}[theorem]{Corollary}
\newtheorem{lemma}[theorem]{Lemma}

\theoremstyle{definition}
\newtheorem{remark}[theorem]{Remark}
\newtheorem{definition}{Definition}
\newtheorem{example}{Example}

\newcommand{\vep}{\varepsilon}

\newcommand{\diam}{\operatorname{diam}}
\renewcommand{\epsilon}{\varepsilon}

\newcommand{\per}{\operatorname{Per}}

\newcommand{\intt}{\operatorname{int}}

\newcommand{\Hom}{\operatorname{Hom}}

\newcommand{\supp}{\operatorname{supp}}

\newcommand{\diff}{\operatorname{Diff}}
\newcommand{\ind}{\operatorname{ind}}

\newcommand{\Diff}{{\rm Diff}^1}

\begin{document}

\title[On sufficient conditions for the transitivity]{On sufficient conditions for the transitivity of homeomorphisms}
%Anosov diffeomorphisms]{On the transitivity of Anosov diffeomorphisms
%\title[.]{.}

\author{M. Carvalho}
\author{V. Coelho}
\author{L. Salgado}
%\author{P. Varandas}

\address{Maria Carvalho, CMUP \& Departamento de Matem\'atica, Faculdade de Ci\^encias da Universidade do Porto \\
Rua do Campo Alegre 687, 4169-007 Porto, Portugal.}
\email{mpcarval@fc.up.pt}

\address{Vin\'icius Coelho, Universidade Federal do Oeste da Bahia, Centro Multidisciplinar de Bom Jesus da Lapa\\
Av. Manoel Novais, 1064, Centro, 47600-000 - Bom Jesus da Lapa-BA-Brazil}
\email{viniciuscs@ufob.edu.br}

\address{Luciana Salgado, Universidade Federal do Rio de Janeiro, Instituto de Matem\'atica\\
Avenida Athos da Silveira Ramos 149 Cidade Universit\'aria, P.O. Box 68530,
21941-909 Rio de Janeiro-RJ-Brazil }
\email{lsalgado@im.ufrj.br, lucianasalgado@ufrj.br}

\subjclass[2010]{Primary: 37B65, 37C20, 37C75, 37D20.}

\keywords{Anosov diffeomorphism; Shadowing; Specification; Stability; Transitivity}

%%%%%%%%%%%%%%%%%%%%%%%%%%%%%%%%%%%%%%%%%%%%%%%%%%%%%%%%%%%%%%%%%%%%%%%%%%%%%%%%%%%%%%%%%%%%%

\begin{abstract}
We derive sufficient conditions for a homeomorphism of a compact metric space without isolated points to be topologically transitive. In particular, we prove that a homeomorphism with the shadowing property has a dense orbit if either it has the su-intersecting property in the whole space; or it has the barycenter property in the non-wandering set; or else it is accessible in the recurrent set. To elucidate their dynamical nature, we compare these three conditions, discuss examples with a variety of dynamics and present some applications of interest.

\end{abstract}

%%%%%%%%%%%%%%%%%%%%%%%%%%%%%%%%%%%%%%%%%%%%%%%%%%%%%%%%%%%%%%%%%%%%%%%%%%%%%%%%%%%%%%%%%%%%%

\date{\today}
\maketitle

%%%%%%%%%%%%%%%%%%%%%%%%%%%%%%%%%%%%%%%%%%%%%%%%%%%%%%%%%%%%%%%%%%%%%%%%%%%%%%%%%%%%%%%%%%%%%

\section{Introduction} \label{sec:statement-result1}

Let $M$ be a $C^\infty$ closed (that is, connected, compact, boundaryless, finite dimensional) Riemannian manifold and $\Diff(M)$ denote the set of all $C^1$ diffeomorphisms of $M$ endowed with the $C^{1}$ topology. We say that $f\in \Diff(M)$ is Anosov if $M$ is a hyperbolic set for $f$, meaning that there are constants $c > 0$ and $0 < \lambda < 1$ and a $Df-$invariant continuous splitting $T_x M = E^s_f(x) \oplus E^u_f(x)$ of the tangent space of $M$ at every $x \in M$ so that, for all non-negative integer $n$, one has 

 \begin{align*}
\|Df^n_x(v)\| &\leqslant\,  c\,\lambda^n \, \|v\|  \quad \forall v \in E^s_f(x) \\
\|Df^{-n}_x(v)\| &\leqslant\,  c\,\lambda^n \, \|v\|  \quad \forall v \in E^u_f(x).
\end{align*}

A classical conjecture regarding the classification of Anosov diffeomorphisms states that the non-wandering set of any $C^1$ Anosov diffeomorphism is the whole manifold. If true, then any $C^1$ Anosov diffeomorphism is topologically transitive and its set of periodic points is dense in the manifold.   
(We refer the reader to \cite{ABD2004} for a more general statement.) And conversely: due to Smale's spectral decomposition \cite{Smale67}, to establish that conjecture it is enough to prove that every Anosov diffeomorphism is topologically transitive. Motivated by a well-known sufficient condition for an Anosov diffeomorphism to be topologically transitive, in this work we address other dynamical assumptions on homeomorphisms with the shadowing property under which we also guarantee the existence of a dense orbit.

Let $(X,d)$ be a compact metric space and $T\colon X \to X$ be a homeomorphism. Given $x \in X$, denote by $\mathcal{O}(x)= \{T^{n}(x)\colon\, n \in \mathbb{Z} \}$ the \emph{orbit} of $x$ by $T$, and by $\overline{\mathcal{O}(x)}$ its closure in $X$. In what follows, $\Omega(T)$ stands for the set of \emph{non-wandering points} of $T$, $\per(T)$ is its subset of \emph{periodic points} and $\per_h(T)$ denotes its subset of \emph{hyperbolic} periodic points when $T$ is a diffeomorphism. Given $x, y \in X$, $n \in \mathbb{N}$ and $\delta > 0$, a $\delta-$chain of $T$ from $x$ to $y$ of length $n$ is a finite collection $x_0 = x, \,x_1,\, \cdots,\, x_n=y$ of points in $X$ such that  
$$d\big(T(x_i),\,x_{i+1}\big) \,\leqslant\, \delta \quad \quad \forall\, i \in \{0, \cdots, n-1\}.$$  A point $x$ in $X$ is \emph{chain-recurrent } if, for any $\delta > 0$, there exists a $\delta-$chain of $T$ from $x$ to $x$. We denote by $CR(T)$ the set of \emph{chain-recurrent points} of $T$. The \emph{recurrent set} of $T$, say $R(T)$, is given by those points in $X$ which are accumulated by their orbits in the future and in the past, that is, $R(T) = \big\{x \in X \colon x \in \alpha(x) \cap \omega(x)\big\}.$ The \emph{stable and unstable sets} of $x \in X$ with respect to $T$ are defined respectively by 
\begin{eqnarray*}
W^{s}_T(x) &=& \big\{y \in X \colon \, \lim_{n \, \to \, +\infty}\, d\big(T^{n}(x),T^{n}(y)\big) = 0\big\} \\
W^{u}_T(x) &=& \big\{y \in X \colon \, \lim_{n \, \to \, +\infty}\, d\big(T^{-n}(x),T^{-n}(y)\big) = 0\big\}.
\end{eqnarray*}
Given $\varepsilon > 0$, the \emph{local stable and unstable sets} with size $\varepsilon$ at $x \in X$ with respect to $T$ are defined, respectively,  by
\begin{eqnarray*}
W^{s}_{T,\,\varepsilon}(x) &=& \big\{y \in X \colon \, d\big(T^{n}(x),T^{n}(y)\big) \, < \, \varepsilon \quad \forall n \in \mathbb{N}\cup \{0\}\big\} \\
W^{u}_{T,\,\varepsilon}(x) &=& \big\{y \in X \colon \, d\big(T^{-n}(x),T^{-n}(y)\big) \, < \, \varepsilon \quad \forall n \in \mathbb{N}\cup \{0\}\big\}.
\end{eqnarray*}
More generally, for $x \in X$ and a closed nonempty set $A \subseteq X$, we consider the usual distance
$$\mathcal{D}(x, A) = \inf\,\big\{d(x,a)\colon \, a \in A\big\}$$
and define the stable and unstable sets of $A$ with respect to $T$ by 
\begin{eqnarray*}
W^{s}_T(A) &=& \big\{y \in X \colon \, \lim_{n \, \to \, +\infty}\, \mathcal{D}\big(T^{n}(y),T^{n}(A)\big) = 0\big\} \\
W^{u}_T(A) &=& \big\{y \in X \colon \, \lim_{n \, \to \, +\infty}\, \mathcal{D}\big(T^{-n}(y),T^{-n}(A)\big) = 0\big\}.
\end{eqnarray*}
The next definitions play a key role in this work.

\begin{definition}
Let $(X,d)$ be a compact metric space, $T \colon X \to X$ be a homeomorphism and $A$ be a nonempty subset of $X$. We say that $T$ has the \emph{$su-$intersecting property in $A$} if $W^{s}_T(a) \cap W^{u}_T(b) \neq \emptyset$ for every $a, b \in A$. The homeomorphism $T$ is said to have the \emph{weak $su-$intersecting property in $A$} if $W^{s}_T\big(\overline{\mathcal{O}_T(a)} \big) \cap W^{u}_T\big( \overline{\mathcal{O}_T(b)} \big) \neq \emptyset$ for every $a,b \in A$.
\end{definition}

Clearly, if an Anosov diffeomorphism $f\colon M \to M$ on a smooth closed Riemannian manifold $M$ has the weak $su-$intersecting property in $\per(f)$, then it is topologically transitive. Indeed, by Smale's spectral decomposition theorem, the set $\Omega(f)$ is a finite union $\Omega_1 \cup \cdots \cup \Omega_K$ of disjoint closed $f-$invariant pieces, the so called basic sets, on each of which $f$ is topologically transitive. Moreover, for each $1 \leqslant j \leqslant K$, each $\Omega_j$ contains a dense subset of periodic orbits and $W^{s}_f(\Omega_j) = \cup_{x \,\in\, \Omega_j} \,W^s_f(x)$, $W^{u}_f(\Omega_j) = \cup_{x \,\in\, \Omega_j} \,W^u_f(x)$ and $M = \cup_{j=1}^K \, W^s_f(\Omega_j) = \cup_{j=1}^K \, W^u_f(\Omega_j)$ (cf. \cite[Proposition~3.10]{Bowen75}). Take a repeller $\Omega_r$ and an attractor $\Omega_a$, whose existence is due to the adapted filtration for the spectral decomposition of $\Omega(f)$ (cf. \cite[\S~2]{Shub}). It is known that $\Omega_r$ contains the stable manifolds of all its points and, similarly, $\Omega_a$ contains the unstable manifolds of its elements (cf. \cite[Lemma~4.9]{Bowen75}). So, if $x \in \Omega_r \cap \per(f)$ and $y \in \Omega_a \cap \per(f)$,  then $W^s_f(x) \subset \Omega_r$, $W^u_f(x) \subset \Omega_a$ and, by the weak $su-$intersecting property in $\per(f)$, one also has $W^s_f(x)\cap W^u_f(y) \neq \emptyset$. Hence $\Omega_r \cap \Omega_a \neq \emptyset$, which is a contradiction unless $f$ has only one basic set equal to $M$. Thus $f$ is topologically transitive.

In \cite{ABC11}, F. Abdenur, C. Bonatti and S. Crovisier introduced another property on periodic orbits which is relevant in this context, and inspired a slightly different concept suggested by W. Sun and X. Tian in \cite{ST12}.  Here we explore the following generalization.

\begin{definition}\label{BP-definition}
Let $(X,d)$ be a compact metric space, $T \colon X \to X$ be a homeomorphism and $A \subseteq X$ be a nonempty set. We say that $T$ has the \emph{barycenter property in $A$} if for every $p,\, q \in A$ and any $\varepsilon >0$, there is a positive integer $N = N(\varepsilon, p, q)$ such that, for any $n_1, n_2 \in \mathbb{N}$, one can find an integer $m \in [0, N]$ and $x_0 \in X$ satisfying $d\big(T^i(x_0), T^i(p)\big) < \varepsilon$ for every $-n_1 \leqslant  i \leqslant 0$ and $d\big(T^{i+m}(x_0), T^i(q)\big) < \varepsilon$ for every $0 \leqslant  i \leqslant n_2$.
\end{definition}

It turns out that, if $T$ has the weak $su-$intersecting property in $\per(T)$, then $T$ has the barycenter property in $\per(T)$ (cf. \cite[Lemma~2.2]{ST12}). Moreover, if $X$ has no isolated points and $T$ is expansive, then $T$ has the weak $su-$intersecting property in $\per(T)$ if and only if $T$ has the barycenter property in $\per(T)$. Indeed, the proof of the aforementioned Lemma~2.2 in \cite{ST12} only requires that for every periodic point $p$ of $T$ there exists $\varepsilon > 0$ such that $W^{s}_{T,\,\varepsilon}(p) \subseteq W^{s}_{T}(p)$  and $W^{u}_{T,\,\varepsilon}(p) \subseteq W^{u}_{T}(p)$; and, surely, this condition is valid when $T$ is expansive (cf.  \cite[Lemma~1]{M79}). This ultimately implies that, within the setting of $C^1$ diffeomorphisms of a closed Riemannian manifold whose periodic points are hyperbolic, the barycenter property in the set of periodic points is equivalent to the weak $su-$intersecting property in that set. For instance,  within the family of $C^1$ $\Omega-$stable diffeomorphisms we may equally use one or the other property.

In \cite[Definition~1.1]{Brin75}, M. Brin proposed another notion which proved pertinent to the pursuit of the topological transitivity within Anosov diffeomorphisms. We will investigate a related concept.

\begin{definition}\label{def:accessible} Let $(X,d)$ be a compact metric space, $T\colon X \to X$ be a homeomorphism and $A \subseteq X$ be a nonempty set. Given $x, y \in A$, a $su-$path in $A$ from $x$ to $y$ is a finite collection of points $z_{0}, \cdots, z_{m} \in A$ such that $z_{0}=x$, $z_{m}=y$ and $z_{i} \in W^{s}(z_{i-1}) \cup W^{u}(z_{i-1})$ for every $i \in \{1,\cdots,m\}.$ We say that $T$ is \emph{accessible in $A$} if for any $x,y \in A$ there exists a $su$-path in $A$ from $x$ to $y$. 
\end{definition}

We observe that, if $T$ has the $su-$intersecting property in $A$, then $T$ is accessible in $A$. From \cite[Theorem~1.1]{Brin75}, we easily deduce that if $f\colon M \to M$ is a partially hyperbolic diffeomorphism of a closed Riemannian manifold such that $\Omega(f)=M$ and $f$ is accessible in $M$, then $f$ is topologically transitive. On the other hand, D. Dolgopyat and A. Wilkinson showed in \cite{DW} that there is a $C^1$ open dense subset $\mathcal{A}$ of the space of strong partially hyperbolic diffeomorphisms of $M$ such that every $f \in \mathcal{A}$ is accessible in $M$. However, accessibility in the whole manifold is not sufficient for topological transitivity: V. Nitica and A. Torok constructed in \cite{NT} an open set of diffeomorphisms which are accessible in the whole manifold but fail to be topologically transitive.

Our first result concerns homeomorphisms with the shadowing property (see the definition in Subsection~\ref{def-two}), linking the three previous definitions with the existence of a dense orbit.

\begin{maintheorem}\label{maintheo-weak-hmp}
Let $T\colon X \to X$ be a homeomorphism of a compact metric space $(X,d)$ without isolated points, and consider $\emptyset \neq A \subseteq X$. 
\begin{itemize}
\item[(a)] If $\overline{A} = X$ and $T$ has the barycenter property in $A$, then $T$ is topologically transitive.

\item[(b)] If $T(A) = A$, $\Omega(T) = \overline{A}$, $T$ has the barycenter property in $A$ and $\Omega(T) = CR(T)$, then $T$ is topologically transitive.

\end{itemize}
Suppose that $T$ has the shadowing property. Then: 
\begin{itemize}
\item[(c)] If $T$ is topologically transitive, then $T$ has the barycenter property in $A$.

\item[(d)] If $T(A) = A$, $\Omega(T) = \overline{A}$ and $T$ has the barycenter property in $A$, then $T$ is topologically transitive.

\item[(e)] If $T$ is accessible in $R(T)$, then $T$ is topologically transitive.

\item[(f)] If $T$ has the $su-$intersecting property in $X$, then $T$ is topologically mixing.
\end{itemize}
\end{maintheorem}

We observe that the assumption $\Omega(T) = CR(T)$ holds for homeomorphisms with the shadowing property \cite[Theorem~3.1.2]{AH94}; generically in $\Diff(M)$ \cite[Corollaire~1.2]{BC04}; and by the Closing Lemma \cite{Pugh}, for $\Omega$-stable diffeomorphisms if $M$ has dimension $\geqslant 2$. However, the condition $\Omega(T) = CR(T)$ is not equivalent to shadowing, as exemplified by an irrational rotation on the unit circle, which satisfies $\Omega(T) = CR(T)$ but does not exhibit shadowing.

Every $C^1$ Anosov diffeomorphism is structurally stable, has the shadowing property, is expansive, and its non-wandering set is hyperbolic and coincides with the closure of the set of periodic points (cf. \cite{Anosov67, Bowen75, Pugh}).  Therefore, we recover from the previous items (b) and (c) the following known result: if $f\colon M \to M$ is an Anosov diffeomorphism of a smooth closed Riemannian manifold $M$  
which is either accessible in $R(f)$ or has the barycenter property in $\per(f)$, then $f$ is a topologically transitive Anosov diffeomorphism (since $\Omega(f) = \overline{\per(f)}$). Whereas, from item (a) we conclude that, conversely, if $f\colon M \to M$ is a topologically transitive Anosov diffeomorphism, then it has the barycenter property on $M$, since $\Omega(f) = M$ and $f$ has the shadowing property on $M$.

In Section~\ref{se:examples}, the reader may find examples of homeomorphisms which either have the $su-$intersecting property in $X$ but do not satisfy the shadowing property, or with the shadowing property but failing to satisfy the $su-$intersecting property in $X$.

In \cite{L16-2}, M. Lee proved that the $C^{1}$ interior of the diffeomorphisms satisfying the barycenter property in $\per(f)$ is contained in the set of $\Omega-$stable diffeomorphisms; and that, generically, if a $C^1$ diffeomorphism has the barycenter property in $\per(f)$ then it is $\Omega-$stable. We will improve this information by showing that the $C^{1}$ interior of diffeomorphisms satisfying the weak $su-$intersecting property in $\per_h(f)$ is equal to the set of topologically transitive Anosov diffeomorphisms (see Proposition~\ref{theo-c1-interior}); and by proving that, generically, if a $C^1$ diffeomorphism has the weak $su-$intersecting property in 
$\per_h(f)$, then it is a topologically transitive Anosov diffeomorphism (see Subsection~\ref{sse:generic}).
Assuming the shadowing property and expansivity, we obtain the following correlation between the previous definitions and the specification property (see the definition in Subsection~\ref{def-two}).

\begin{maintheorem}\label{maintheo-Homeo-hyperbolicmeet}
Let $(X,d)$ be a compact metric space without isolated points and $T\colon X \to X$ be an expansive homeomorphism which satisfies the shadowing property. Then the following statements are equivalent:
\begin{itemize}
\item[$(a)$] $T$ has the $su-$intersecting property in $\per(T)$.

\item[$(b)$] $T$  has the specification property.
\end{itemize}
If, in addition, $X$ is connected, then the following statements  are equivalent:
\begin{itemize}
\item[$(c)$] $T$ has the $su-$intersecting property in $X$.

\item[$(d)$] $T$ has the barycenter property in $\per(T)$.
\end{itemize} 
\end{maintheorem}

We stress that both Theorems~\ref{maintheo-weak-hmp} and  \ref{maintheo-Homeo-hyperbolicmeet} hold without any assumption on the hyperbolicity of the set of periodic points. Their conclusions seem to hint that the barycenter property in the set of periodic points is in general weaker than the $su$-intersecting property in that set. For further discussion on this issue, we refer the reader to Example~\ref{ex:bp-not-shadowing}.

In several instances of this work it is necessary to know that, for homeomorphisms on compact metric spaces, topological transitivity is equivalent to the existence of a dense orbit (cf. \cite[Theorem~5.8]{W78}). Yet, it is also fundamental to relate one-sided topological transitivity to topological transitivity. These concepts are equivalent if the underlying compact metric space $X$ does not contain isolated points (cf. \cite{Silv} and \cite[Theorem~1.4]{AC12}); or else if the non-wandering set is the entire space $X$ (cf. \cite[Theorem~5.10]{W78}). Since the former condition concerns the phase space, whereas the latter requests prior information on the dynamics, we opted for the former, which is why it is included in the statements of the foregoing theorems.

For any $C^1$ $\Omega-$stable diffeomorphism, the next result establishes a general equivalence between the topological transitivity and several other dynamical properties.

\begin{maintheorem}\label{theo-aux-complet}
Let $M$ be a smooth closed Riemannian manifold, with dimension at least $2$, and $f\colon M \to M$ be a $C^1$ $\Omega-$stable diffeomorphism. The following statements are equivalent:
\begin{itemize}
\item[$(1)$]  $f$ has the two-sided limit shadowing property.
\item[$(2)$]  $f$ has the $su-$intersecting property in $M$.
\item[$(3)$]  $f$ has the $su-$intersecting property in $\per(f)$.
\item[$(4)$]  $f$ has the weak $su-$intersecting property in $\per(f)$.
\item[$(5)$]  There exists $A \subseteq M$ such that $f(A)=A$, $\Omega(f) = \overline{A}$ and $f$ has the barycenter property in $A$.
\item[$(6)$]  $f$ is a topologically transitive Anosov diffeomorphism.
\item[$(7)$]  $f$ has the specification property.
\item[$(8)$]  $f$ has the almost specification property.
\item[$(9)$]  $f$ has the asymptotic average shadowing property.
\item[$(10)$] $f$ has the average shadowing property.
\item[$(11)$] $f$ is chain-transitive.
\item[$(12)$] $\mathcal{P}_{f}^{erg}(M)$ is dense in $\mathcal{P}_{f}(M)$. 
\item[$(13)$] $f$ has the gluing orbit property.
\end{itemize}
\end{maintheorem}

Some comments are in order regarding other known sufficient conditions for an Anosov diffeomorphism to be topologically transitive, and how they compare with the $su-$intersecting property. In \cite{Brin77}, M. Brin proved that, if the local product structure of a $C^1$ Anosov manifold can be extended to the whole manifold 
(see the precise meaning of this wording in \cite[page~12]{Brin77}), then it is topologically transitive. Actually, Brin's condition implies that the diffeomorphism has the $su-$intersecting property in the whole manifold, as shown in \cite[Remark~4]{Brin77}. Another sufficient condition was proposed by B. Carvalho in \cite{C15}: if a $C^1$ Anosov diffeomorphism has the two-sided limit shadowing property (see the definition in Subsection~\ref{def-two}), then it is topologically transitive. Moreover, \cite[Lemma~2.4]{C15} ensures that if a homeomorphism  acting on a compact metric space has the two-sided limit shadowing property, then it has the $su-$intersecting property on the whole space. In Section~\ref{se:examples}, we present an example of a homeomorphism that has the $su-$intersecting property on the whole domain although it does not satisfy the two-sided limit shadowing property. In \cite{Micena22, MicenaHertz23}, we find conditions of a different nature which also ensure the topological transitivity of an Anosov diffeomorphism. On the one hand, it is known that transitive Anosov diffeomorphisms have a unique measure of maximal entropy (cf. \cite{Bowen75}); F. Micena and J. Rodriguez-Hertz proved in  \cite{MicenaHertz23} that, conversely, if a $C^1$ Anosov diffeomorphism has a unique measure of maximal entropy and every point in the manifold is regular in the sense of Oseledets's theorem \cite{Oseledets}, then the diffeomorphism is topologically transitive. On the other hand, building upon \cite[Theorem~4.14]{Bowen75}, F. Micena \cite{Micena22} linked the topological transitivity of any $C^2$ Anosov diffeomorphism to the Jacobean of the derivative on all the periodic points through the existence of an invariant probability measure equivalent to the volume on the manifold. Micena's condition is too restrictive, though, since it is known that among $C^2$ Anosov diffeomorphisms the ones that admit no invariant probability measure equivalent to the volume on the manifold form an open dense set (cf. \cite[Corollary~4.15]{Bowen75}).

The paper is organized as follows. In Section~\ref{se:def}, we set some notation and basic definitions. In Section~\ref{se:examples} we include concrete examples which show how to construct homeomorphisms with or without the $su-$intersecting property. Sections~\ref{se:proof-thmA}, \ref{se:proof-thmB} and \ref{se:proof-thmC} consist of the proofs of the main results, together with a thorough discussion about some dynamical information conveyed by the $su-$intersecting and the barycenter properties.  Section~\ref{se:appl} is made of applications.

\section{Definitions}\label{se:def}

In this section we include a brief glossary of the main concepts we will use.

Let $(X,d)$ be a compact metric space and $T\colon X \to X$ be a homeomorphism. A point $p \in X$ is called non-wandering for $T$ if for every neighborhood U of $p$ there exists $n \in \mathbb{Z}\setminus\{0\}$ such that $T^{-n} (U) \cap U \neq \emptyset$. 
A point $p$ in $X$ is periodic for $T$ if there exists $n \in \mathbb{N}$ such that $T^{n}(p)=p$.

\subsection{Topological transitivity}

A homeomorphism $T$ is topologically transitive if for every nonempty open sets $U,\,V \subset X$ there exists $n \in \mathbb{Z}$ such that $T^{n}(U)\cap V \neq \emptyset$. It is known that this is equivalent to say that $T$ has a dense orbit, that is, there is some $x \in X$ with $\{T^n(x) \colon n \in \mathbb{Z}\}$ dense in $X$ (see \cite[Theorem~1.4]{AC12} or \cite[Theorem~5.8]{W78}).

\subsection{One-sided topologically transitive}

A homeomorphism $T$ is one-sided topologically transitive if for every nonempty open sets $U,\,V \subset X$ there exists $n \in \mathbb{N}$ such that $T^{n}(U)\cap V \neq \emptyset$. According to \cite[Theorem~5.10]{W78}, a homeomorphism $T$ of a compact metric space $X$ is one-sided topologically transitive if and only if $T$ is topologically transitive and $\Omega(T) = X$.

\subsection{Topological mixing}

A homeomorphism $T$ is called \emph{topologically mixing} if for any two open nonempty sets $U,V \subset X$ there exists a positive integer $N=N(U,V)$ such that for any $n>N$ the intersection $T^n(U) \cap V$ is nonempty. Obviously, topological mixing implies topological transitivity.

\subsection{Chain-transitivity}

A map $T$ is said to be \emph{chain-recurrent} if for any $\delta >0$ and any $x \in X$ there is a $\delta-$chain from $x$ to itself. $T$ is called \emph{chain-transitive} if for any $\delta >0$ and any two points $x, y \in X$ there is a $\delta$-chain from $x$ to $y$.

\subsection{Chain-mixing}

A homeomorphism $T$ is \emph{chain-mixing} if given $\delta>0$ there exists $N \in \mathbb{N}$ such, for every $x,y \in X$ and any integer $n \geqslant N$, there is a $\delta-$chain from $x$ to $y$ of length $n$.

\subsection{Expansiveness}

A homeomorphism $T$ is expansive if there is a constant $c > 0$ such that
$$d\big(T^n(x), T^n(y)\big) \,<\, c \quad \forall \, n \in \mathbb{Z} \quad \quad \Rightarrow \quad \quad x\,=\,y.$$

\subsection{Several concepts of shadowing}\label{def-two}

For further use, we now address distinct notions of shadowing.

\subsubsection{\emph{\textbf{Shadowing}}} Given $\delta > 0$, a sequence $\{x_n\}_{n\,\in\,\mathbb{Z}}$ in $X$ is called a $\delta-$pseudo-orbit of $T$ if
$$d(T(x_n),\,x_{n+1}) \,<\, \delta \quad \forall \, n \in \mathbb{Z}.$$
Given $\varepsilon > 0$, a $\delta-$pseudo-orbit $\{x_n\}_{n\,\in\,\mathbb{Z}}$ is said to be $\varepsilon-$shadowed if there exists $y \in X$ such that
$$d(T^n(y),\,x_n) \,<\, \varepsilon \quad \forall\, n \in \mathbb{Z}.$$
A homeomorphism $T$ has the shadowing property if for any $\varepsilon > 0$ there exists $\delta>0$ such that any $\delta-$pseudo-orbit is $\varepsilon-$shadowed.

\subsubsection{\emph{\textbf{Positive limit shadowing}}} A sequence $(x_n)_{n\,\in\,\mathbb{Z}}$ in $X$ is called a positive limit pseudo-orbit if 
$$\lim_{n \, \to \, +\infty}\,d(T(x_n), x_{n+1}) \, = \, 0.$$
The sequence $\{x_n\}_{n \, \in \, n\mathbb{Z}}$ is said to be positive limit shadowed if there exists $y \in X$ such that 
$$\lim_{n \, \to \, +\infty}\,d(T^n(y),x_n) \,= \,0.$$ 
We say that $T$ has the positive limit shadowing property if every positive limit pseudo-orbit in $X$ is positive limit shadowed. The notion of negative limit shadowing property is similarly defined.

\subsubsection{\emph{\textbf{Two-sided limit shadowing}}} A sequence $(x_n)_{n \,\in\, \mathbb{Z}}$ in $X$ is a two-sided limit pseudo-orbit if 
$$\lim_{n \, \to \, +\infty} \,d(T(x_n),x_{n+1}) \,=\, 0 \,= \,\lim_{n \, \to \, -\infty} \,d(T(x_n),x_{n+1}).$$
A two-sided limit pseudo-orbit is said to be two-sided limit shadowed if there is a point $y\in X$ such that
$$\lim_{n \, \to \, +\infty} \,d(T^n(y),x_n) \, =\,0 \, = \, \lim_{n \, \to \, -\infty} \, d(T^n(y),x_n).$$
A homeomorphism $T$ has the two-sided limit shadowing property if every two-sided limit pseudo-orbit in $X$ is two-sided limit shadowed.

\subsubsection{\emph{\textbf{Average shadowing}}} Given $\delta>0$, a sequence $(x_n)_{n \, \in \, \mathbb{Z}}$ in $X$ is a 
$\delta-$average pseudo-orbit of $T$ if there is a positive integer $N_0$ such that, for every integer $n \geq N_0$ and all integer $k \geq 0$, one has
$$\frac{1}{n}\,\sum_{j=0}^{N-1} \,d(T(x_{j+k}), x_{j+k+1}) \,<\,\delta.$$
Given $\varepsilon>0$, a sequence $(x_n)_{n \, \in \, \mathbb{Z}}$ in $X$ is said to be $\varepsilon-$shadowed in average if there is $y \in X$ such that
$$\limsup_{n\, \to\, +\infty} \,\frac{1}{n}\,\sum_{j=0}^{n-1} d(T^j(y),x_j)\,< \,\varepsilon.$$
A homeomorphism $T$ has the average shadowing property if for every $\varepsilon>0$ there is $\delta>0$ such that every $\delta-$average pseudo-orbit of $T$ is $\varepsilon-$shadowed in average.

\subsubsection{\emph{\textbf{Asymptotic average shadowing}}} A sequence $(x_n)_{n \, \in \, \mathbb{Z}}$ in $X$ is an asymptotic average pseudo-orbit of $T$ provided that 
$$\lim_{n \, \to \, +\infty}\,\frac{1}{n}\,\sum_{j=0}^{n-1}d(T(x_j),x_{j+1})=0.$$
The sequence $(x_n)_{n \, \in \, \mathbb{Z}}$ is said to be asymptotically shadowed in average in $X$ if there is $y\in X$ such that
$$\lim_{n\, \to \, +\infty}\,\frac{1}{n}\,\sum_{j=0}^{n-1}\,d(T^j(y),x_j)=0.$$
A homeomorphism $T$ has the asymptotic average shadowing property if every asymptotic average pseudo-orbit of $T$ in $X$ is asymptotically shadowed in average.

\subsubsection{\emph{\textbf{Specification}}}

Given $n, N \in \mathbb{N}$, a family of $n$ orbit segments of $T$  
$$\xi = \Big\{T^{[a_j, b_j]}(x_j)\colon \,\, j \in \{1, \cdots, n\}\Big\}$$ 
is an $N-$spaced specification if $a_i - b_{i-1} \geqslant N$ for every $2 \leqslant i \leqslant n$. Given $\varepsilon > 0$, we say that a specification $\xi$ is $\varepsilon-$traced by $y \in X$ if
$$d\big(T^k(y), T^k(x_i)\big)\,< \, \varepsilon \quad \quad \forall\, a_i \,\leqslant\, k \,\leqslant\, b_i \quad \forall\, 1 \,\leqslant\, i \,\leqslant\, n. $$
The map $T$ has the specification property if for any $\varepsilon > 0$ there is $N=N(\varepsilon) \in \mathbb{N}$ such that any $N-$spaced specification $\xi$ is $\varepsilon-$traced by some $y \in X$.

\subsection{Gluing orbit property}

The next notion first appeared in \cite{ST12}, where it was called \emph{transitive specification}. It was in \cite{BV19}  that the name gluing orbit property started being used. We  refer the reader to \cite{S19} for useful remarks regarding  this concept.

A continuous map $T\colon  X \to X$ on a compact metric space $X$ satisfies the \emph{gluing orbit property} if for every $\varepsilon>0$ there exists an positive integer $N=N(\vep)$ so that, for any finite sample of points $x_1, x_2, \dots, x_k \in X$ and any positive integers $n_1, \dots, n_k$, there are positive integers $p_1, \dots, p_{k-1} \leqslant N(\vep)$ and a point $x \in X$ such that
\begin{eqnarray*}
d\big(T^j(x),\,T^j(x_1)\big) & \leqslant & \varepsilon \quad \quad \forall\, 0 \leqslant j \leqslant n_1\\
d\big(T^{j\,+\,n_1\,+\,p_1\,+\,\dots\, + \,n_{i-1}\,+\,p_{i-1}}(x),\; T^j(x_i)\big)&  \leqslant & \varepsilon \quad \quad \forall \, 2\leqslant i\leqslant k \quad \forall\, 0\leqslant j\leqslant n_i.
\end{eqnarray*}
A homeomorphism $S\colon  Y \to Y$ on a compact metric space $Y$ has the gluing orbit property if both $S$ and $S^{-1}$ comply with it.

\subsection{Measure center}

In what follows, $\mathcal{P}(X)$ stands for the space of Borel probability measures in $X$ with the weak$^*-$topology, $\mathcal{P}_{T}(X)$ is its subspace of $T-$invariant elements and $\mathcal{P}_{T}^{erg}(X)$ is the subset of ergodic measures of $\mathcal{P}_{T}(X)$.
The measure center of $T$ is the set
$$\mathcal{M}(X,T) \,=\, \overline{\bigcup_{\mu \,\in\,  \mathcal{P}_T(X)} \supp (\mu)}$$
where $\supp (\mu)$ denotes the support of the probability measure $\mu$.

\subsection{Hyperbolicity}

Given a $C^1$ diffeomorphism $f\colon M \to M$ of a closed Riemannian manifold $M$, a periodic point $p$ of $f$ with minimal period $N$ is said to be hyperbolic if the spectrum of the derivative $Df^N_p \colon T_pM \to T_pM$ at $p$ is disjoint from  the unit circle in the complex plane.

\subsection{Stability}

We say that a $C^1$ diffeomorphism $f\colon M \to M$ is structurally stable if there is a $C^1$ neighborhood $\mathcal{U}$ of $f$ such that for any $g \in \mathcal{U}$ there is a homeomorphism $h\colon M \to M$ such that $h \circ f = g \circ h$. If the topological conjugacy with neighbour diffeomorphisms is only valid within the non-wandering set, that is, there is a $C^1$ neighborhood $\mathcal{U}$ such that for any $g \in \mathcal{U}$ there is a homeomorphism $h\colon \Omega(f) \to \Omega(g)$ such that $h \circ f_{|\Omega(f)} = g_{|\Omega(g)} \circ h$, then $f$ is said to be $\Omega-$stable.

\section{Examples}\label{se:examples}

In this section we discuss examples to illustrate the dynamical nature of $su-$intersecting and the barycenter properties. A number of them hint that the assumptions in our results are somehow sharp.

\begin{example}\label{HMP-withoutPOTP-2SLSP}

We start with an example of a homeomorphism on a compact metric space which has the $su-$intersecting property in the whole domain but fails to satisfy both the two-sided limit shadowing and the shadowing property. Consider a two-sided shift $\sigma$ on a space with a finite alphabet endowed with the usual metric   which has the specification property but is not a shift of finite type. Then $\sigma$ has the $su-$intersecting property in the whole space (cf. \cite[Proposition~5]{C15TS}), even though $\sigma$ fails to satisfy both the two-sided limit shadowing property (cf. \cite[Corollary~6]{C15TS} and  the shadowing property (cf. \cite{Walters}).

\end{example}

%%%%%%%%%%%%%%%%%%%%%%%%%%%%%%%%%%%%%%%%%%%%%%%%%%%%%%%%%%%%%%%%%%%%%%%%%%%%%%%%%%%%%%%%%%%%%

\begin{example}\label{Homeo-generic-barycenter}

Let $(X,d)$ be a compact connected manifold, with or without boundary, of dimension $\geqslant 2$, endowed with a distance $d$. A \emph{Lebesgue-like measure} $\mu$ on $X$ is a Borel probability measure 
such that
\begin{itemize}
\item[(i)]  $\,\,\forall\, p \in X \quad \mu(\left\{ p \right\} ) = 0$;
\item[(ii)] $\,\,\mu(U) > 0$ for every nonempty open set  $U \subset X$;
\item[(iii)] $\,\, \mu\,(\partial X) = 0$.
\end{itemize}
Suppose that $X$ possesses a Lebesgue-like measure, say $\mu_0$. Denote by $\Hom(X,\mu_0)$ the set of homeomorphisms of $X$ which preserve $\mu_0$. This space is metrizable by the uniform distance for homeomorphisms, given by $D(f,g) = D_1(f,g) + D_1(f^{-1},g^{-1})$, where $D_1(f,g) = \max_{x \,\in\, X} d(f(x), g(x))$. Moreover, $(\Hom(X,\mu_0), D)$ is a Baire space.

According to \cite[Theorem~1.3, Corollary~1.4]{GL18}, a generic element in $\Hom(X,\mu_0)$ is topologically mixing and satisfies the shadowing property. Using Theorem~\ref{maintheo-weak-hmp}(c), we conclude that, $C^0$ generically in $\Hom(X,\mu_0)$, a homeomorphism has the barycenter property in $A$ for every nonempty subset $A$ of $X$.

\end{example}

\begin{example}\label{POPT-BP-NOTHMP}

We now provide an example of a homeomorphism of a compact metric space which has the barycenter property in the whole domain but fails to satisfy the $su-$intersecting property in that set. Consider the two-sided shift $\sigma \colon X \to X$ on the space $X = Y^\mathbb{Z}$ endowed with the usual metric, where $Y=\{1,2,3,4\}$. Take the matrix $A=\big(a_{i, \, j}\big)_{i, j \, \in \, \{1,2,3,4\}}$ given by
$$A \, = \, \left(\begin{array}{cccc}
0 &1 &0 &0 \\
1 &0 &1 &0 \\
0 &1 &0 &1 \\
0 &0 &1 &0
\end{array}
\right)
$$
and let $\Sigma_A \subset Y^\mathbb{Z}$ be the space of $A-$admissible sequences, that is,
$$\Sigma_A \, = \, \big\{(x_k)_{k \, \in \, \mathbb{Z}} \colon \, a_{x_k\, x_{k+1}} \,=\, 1 \quad \forall k \in \mathbb{Z}\big\}.$$
The matrix $A$ is irreducible (to check this it is enough to compute $A^2$ and $A^3$) but it is not aperiodic. Therefore, the homeomorphism $\sigma_{|\Sigma_A}\colon  \Sigma_A \to \Sigma_A$  is a shift of finite type (hence it has the shadowing property),  is topologically transitive (cf. \cite[Theorem~1.19]{W78}), but is not topologically mixing (cf. \cite[Theorem~1.31]{W78}). Moreover, the space $\Sigma_A$ is perfect (that is, it is closed, totally disconnected and has no isolated points) due to  the fact that the matrix $A$ has lines whose entries sum is equal to $2$ (cf. \cite[Proposition~2.8]{Rob}). Therefore, Theorem~\ref{maintheo-weak-hmp}(c) indicates that $\sigma_{|\Sigma_A}$ has the barycenter property in $\Sigma_A$, whereas Theorem~\ref{maintheo-weak-hmp}(e) indicates that $\sigma_{|\Sigma_A}$ fails to satisfy the $su-$intersecting property in $\Sigma_A$.  

\end{example}

%%%%%%%%%%%%%%%%%%%%%%%%%%%%%%%%%%%%%%%%%%%%%%%%%%%%%%%%%%%%%%%%%%%%%%%%%%%%%%%%%%%%%%%%%%%%%

\begin{example}\label{ex:bp-not-shadowing}
The next example is mentioned in \cite[Remark~1.1]{L16-2} to justify that the barycenter property in $\per(f)$, though a kind of tracing condition, is ultimately distinct from the specification property. Let $f\colon S^1 \to S^1$ be a $C^1$  diffeomorphism whose non-wandering set reduces to two fixed points which are saddle-nodes. Then $f$ does not have the shadowing property (cf. \cite[Theorem~3.1.1]{Pilyugin99}) and is not expansive (cf. \cite[Theorem~5.27]{W78}). Moreover, $f$ is not topologically transitive, so it does not have the specification property (cf. \cite[Proposition~21.3]{DGS}. However, $f$ has the $su-$intersecting property in $\per(f)$, hence the barycenter property in $\per(f)$.

We observe that, on the contrary, if $T\colon X \to X$ is a homeomorphism of a compact metric space $(X,d)$ without isolated points, $A$ is a nonempty subset of $X$ and $T$ has the specification property, then $T$ has the barycenter property in $A$. Indeed, take $p, \,q \in A$ and fix $\varepsilon >0$. Since $T$ has the specification property there exists $N=N(\varepsilon) \in \mathbb{N}$ such that any $N-$spaced specification $\xi$ is $\varepsilon-$traced by some $y \in X$. Given $n_1, n_2 \in \mathbb{N}$, define 
$$a_{1} = -n_{1}, \quad b_{1} = 0, \quad x_{1} = p, \quad  a_{2} = N, \quad b_{2} = n_{2}+N, \quad x_{2} = T^{-N}(q).$$
Note that $a_{2} - b_{1} = N$, so 
$$\xi = \Big\{T^{[a_j, b_j]}(x_j)\colon \,\, j \in \{1, 2\}\Big\}$$ 
is an $N-$spaced specification. Then $\xi$ is $\varepsilon-$traced by some $y \in X$, that is,
\begin{eqnarray*} 
d\big(T^k(y),\, T^k(p)\big) &<& \varepsilon \quad \quad \, -n_{1} \,=\, a_1 \,\leqslant\, k \,\leqslant\, b_1\, = \,0 
\\
d\big(T^\ell(y), \,T^\ell(T^{-N}(q))\big) &<& \varepsilon \quad \quad \, N \,=\, a_2 \,\leqslant\, \ell\, \leqslant \,b_2\, =\, n_{2}+N.
\end{eqnarray*}
Select $j = \ell-N$. Thus, 
$$d\big(T^{j+N}(y), \,T^j (q))\big)\,< \, \varepsilon \quad \quad \, 0 \leqslant j \leqslant n_{2} .$$
So, 
$T$ has the barycenter property in $A$.

\end{example}

\begin{example}\label{MS}

A Morse-Smale diffeomorphism of the unit circle is a simple example of an $\Omega-$stable diffeomorphism of a smooth closed Riemannian manifold which has the shadowing property (cf. \cite{Pilyugin99}), is not expansive (cf. \cite[Theorem~5.27]{W78}) and does not satisfy the barycenter property in the set of periodic points (as is easily concluded from Theorem~\ref{maintheo-weak-hmp}(d)).

\end{example}

%%%%%%%%%%%%%%%%%%%%%%%%%%%%%%%%%%%%%%%%%%%%%%%%%%%%%%%%%%%%%%%%%%%%%%%%%%%%%%%%%%%%%%%%%%%%%

\begin{example}\label{POPT-withoutHMP}

We now construct another example of a homeomorphism on a compact metric space $X$ possessing the shadowing property but not satisfying the $su-$intersecting property in $X$.
\medskip

\noindent \underline{Claim 1}: The identity map on a compact metric space $(X,d)$ has the shadowing property if and only if $X$ is totally disconnected (cf. \cite[Theorem~2.3.2]{AH94}).
\medskip

\noindent \underline{Claim 2}: The $su-$intersecting property is invariant by iteration.

\begin{lemma}\label{lemma-iterations-hyperbolic-meet}
Let $T\colon X \to X$ be a homeomorphism of a compact metric space $(X,d)$ without isolated points and $A \subseteq X$ a nonempty subset of $X$. If $T$ has the $su-$intersecting property in $A$ then $T^{k}$ has the $su-$intersecting property in $A$ for every positive integer $k$.
\end{lemma}

\begin{proof} Fix $k \in \mathbb{N}$ and take $a,b \in A$. Since $T$ has the $su-$intersecting property in $A$, we have $W^s_T(a) \cap W^{u}_T(b) \neq \emptyset$ and $W^{s}_T(b) \cap W^{u}_T(a) \neq \emptyset$. Consider $c \in W^{s}_T(a) \cap W^{u}_T(b)$. Then 
$$\lim_{n \, \to \, +\infty}\,d(T^{n}(a),T^{n}(c)) \,=\,0 \,=\, \lim_{n \, \to \, +\infty}\,d(T^{-n}(b),T^{-n}(c)).$$
Thus, 
$$\lim_{n \, \to \, +\infty}\,d((T^{k})^{n}(a),(T^{k})^{n}(c)) \,=\,0 \,=\, \lim_{n \, \to \, +\infty}\,d((T^{k})^{-n}(b),(T^{k})^{-n}(c)).$$
Hence $c \in W^{s}_{T^{k}}(a) \cap W^{u}_{T^{k}}(b)$. In a similar way we show that $W^{s}_{T^{k}}(b) \cap W^{u}_{T^{k}}(a) \neq \emptyset$.
\end{proof}

\noindent \underline{Claim 3}: The $su-$intersecting property induces some rigidity on the dynamics.

\begin{lemma}\cite[Lemma~3.6]{C15} or \cite[Lemma~4.2]{L16}\label{lemma-chain-transitive}
Let $T\colon X \to X$ be a homeomorphism of a compact metric space $(X,d)$ without isolated points. If $T$ has the $su-$intersecting property in $X$ then $T$ is chain-transitive.
\end{lemma}

A homeomorphism $T\colon X \to X$ of a compact metric space $(X,d)$ is said to be distal if 
$$\inf\,\{d(T^{n}(x), T^{n}(y)) \colon \, n \in \mathbb{Z}\} = 0 \quad \quad \Rightarrow \quad \quad x = y.$$
For instance, the identity map of any space is distal. 

\noindent \underline{Claim 4}: If a homeomorphism is distal and has the shadowing property, then it does not satisfy the $su-$intersecting property in the whole domain.

\begin{lemma}\label{lemma-distal}
Let $(X,d)$ be a compact metric space, containing at least two elements and without isolated points, and $T \colon X \to X$ be a distal homeomorphism. 
\begin{itemize}
\item[(a)] If $T^{k}$ is chain-transitive for every $k \in \mathbb{N}$, then $T$ does not have the shadowing property.
\item[(b)] If $T$ has the shadowing property, then $T$ does not satisfy the $su-$intersecting property in $X$.
\end{itemize}
\end{lemma}

\begin{proof}
Item $(a)$ is Lemma 5.2 in \cite{GU07}. Concerning item $(b)$, assume that $T$ has the shadowing property and, by contradiction, that it also has the $su-$intersecting property in $X$. By Lemma~\ref{lemma-iterations-hyperbolic-meet}, we know that $T^{k}$ has the $su-$intersecting property in $X$ for every positive integer $k$. Using Lemma~\ref{lemma-chain-transitive}, we conclude that $T^{k}$ is chain-transitive for every positive integer $k$. Therefore, by item $(a)$ of Lemma~\ref{lemma-distal}, $T$ cannot have the shadowing property. 
\end{proof}

Consequently, the identity map on a totally disconnected compact metric space $X$, containing at least two elements and without isolated points, has the shadowing property but does not satisfy the $su-$intersecting property in $X$, nor the barycenter property in $X$ according to Theorem~\ref{maintheo-weak-hmp}(d).

On the other hand, the identity map (like any other rational rotation) of the unit circle has a non-wandering set equal to the closure of the set of periodic points, but it does not have the shadowing property nor the barycenter property in that set.

\end{example}

%%%%%%%%%%%%%%%%%%%%%%%%%%%%%%%%%%%%%%%%%%%%%%%%%%%%%%%%%%%%%%%%%%%%%%%%%%%%%%%%%%%%%%%%%%%%%

\begin{example}\label{ex:metade-bp}

Consider the compact set $X =\big\{ 0,1 \big\} \cup \big\{\frac{1}{n}\colon \, n \geqslant 2 \big\} \cup \big\{1-\frac{1}{n}\colon \, n \geqslant 3 \big\}$ with the induced topology as a subset of $\mathbb{R}$. Define the homeomorphism $T\colon X \to X$ by $T(0)=0$, $T(1)=1$, and $T$ maps any other point to the next point on the right. Then
\begin{eqnarray*}
W^s_T(1) &=& \big\{ 1 \big\} \cup \big\{1/n \,\colon \, n \geqslant 2 \big\} \cup \big\{1-1/n \,\colon \, n \geqslant 3 \big\} \quad \quad \text{and} \quad \quad 
W^u_T(1) \,=\, \big\{ 1 \big\} \\
W^u_T(0) &=& \big\{0 \big\} \cup \big\{1/n \,\colon \, n \geqslant 2 \big\} \cup \big\{1-1/n \,\colon \, n \geqslant 3 \big\}\quad \quad \text{and} \quad \quad W^s_T(0) \,=\, \{0\}. 
\end{eqnarray*}
So, $T$ does not have the $su-$intersecting property in $\per(T)$. Moreover, if $p=1$ and $q=0$ then, given $0 < \varepsilon < 1/2$, the subset of the orbit of $1/2$ by $T$ outside the union $[0,\varepsilon[ \,\cup \,\,]1-\varepsilon, 1]$ is finite. Let
\begin{eqnarray*}
x_0 &=&  \max \, \{T^{-n}(1/2) \in [0, \varepsilon[ \,\colon \,n \in \mathbb{N}\}\\
N &=& \min\,\{k \in \mathbb{N}\,\colon T^k(x_0) \in \,\,]1-\varepsilon, 1]\}.
\end{eqnarray*}
Therefore, for any $n_1, n_2 \in \mathbb{N}$, one can find an integer $m = N  \in [0, N]$ and $x_0 \in X$ satisfying $d\big(T^i(x_0), T^i(p)\big) < \varepsilon$ for every $-n_1 \leqslant  i \leqslant 0$ and $d\big(T^{i+m}(x_0), T^i(q)\big) < \varepsilon$ for every $0 \leqslant  i \leqslant n_2$. Thus, the pair $p, q$ complies with the condition of the barycenter property in $\per(T)$. However, if we exchange the fixed points this condition is no longer valid. Thus, $T$ does not have the barycenter property in $\per(T)$. We note that $\Omega(T)=\overline{\per(T)}$, $T$ has the shadowing property (see \cite[Proposition~3]{Pen-Eeu}) and is topologically transitive. However, $X$ has isolated points and the conclusion of Theorem~\ref{maintheo-weak-hmp}(c) does not hold.
\end{example}

%%%%%%%%%%%%%%%%%%%%%%%%%%%%%%%%%%%%%%%%%%%%%%%%%%%%%%%%%%%%%%%%%%%%%%%%%%%%%%%%%%%%%%%%%%%%%

\begin{example}\label{ex:almostA}

The next example calls our attention to the relevance of the uniform hyper\-bolicity assumption in Theorem~\ref{theo-aux-complet}. Let $f\colon M \to M$ be a $C^1$  Anosov diffeomorphism of the $2-$dimensional torus and $g\colon M \to M$ be a $C^1$ almost Anosov diffeomorphism obtained by H. Hu in \cite{Hu} through a first perturbation on $f$. The resulting diffeomorphism $g$ is uniformly hyperbolic except at one fixed point $p$, where $Df_p$ is the identity map. The stable and unstable foliations of $f$ persist under the perturbation, as well as their intersections. So $g$ has the $su-$intersecting property in $\per(g)$, inherited from $f$. By construction, $g$ is topologically transitive, but  is not an Anosov diffeomorphism. A similar construction may be found in \cite{HY}.

Another way to produce this kind of systems is to start with an Anosov diffeomorphism $f$ and bend the unstable manifold of a fixed point $p$ at a heteroclinic intersection point $q$ until the stable and unstable manifolds of $p$ become tangent at $q$, and only at $q$. Such a perturbation was described by H. Enrich in \cite{Enrich}, where a cubic tangency is created, and by A. Gogolev in \cite{Gogolev}, who deals with a transverse quadratic tangency. Again, the resulting diffeomorphism is topologically conjugate to $f$ and all periodic orbits remain hyperbolic, so $g$ is topologically transitive and has the $su-$intersecting property in $\per(g)$. Yet, it is not an Anosov diffeomorphism.

\end{example}

%%%%%%%%%%%%%%%%%%%%%%%%%%%%%%%%%%%%%%%%%%%%%%%%%%%%%%%%%%%%%%%%%%%%%%%

\begin{example}\label{ex:acc-noshad}

Both Theorems~\ref{maintheo-weak-hmp} and \ref{maintheo-Homeo-hyperbolicmeet} assume that the homeomorphism has the shadowing property. The next example indicates that this is an essential hypothesis to obtain topological transitivity.

Recall that, in \cite{Brin75}, M. Brin proved that if $f\colon M \to M$ is a $C^1$ strong partially hyperbolic diffeomorphism of a closed Riemannian manifold such that $\Omega(f)=M$ and $f$ is accessible in $M$, then $f$ is topologically transitive. In order to show that Brin's result cannot be generalized to the case where the chain recurrent set of $f$ is the whole manifold (that is, when $f$ is chain-transitive), S. Gan and Y. Shi constructed in \cite{GS17} a partially hyperbolic diffeomorphism on the $3-$dimensional torus $\mathbb{T}^3$ which is accessible in $M$ and chain-transitive, but fails to be topologically transitive. Notice, however, that, as a chain-transitive diffeomorphism which has the shadowing property is topologically transitive \cite[Lemma~2.1]{L12}, the example of S. Gan and Y. Shi has no such property. 

\end{example}

\section{Proof of Theorem~\ref{maintheo-weak-hmp}}\label{se:proof-thmA}

In this section we shall explain why, under the assumptions of Theorem~\ref{maintheo-weak-hmp}, the topological transitivity implies that the barycenter property is valid in any nonempty subset $A$ of $X$; and, conversely, for what reason the barycenter property in an invariant subset $A$ such that $\Omega(T) = \overline{A}$ yields topological transitivity. 

Let $T\colon X \to X$ be a homeomorphism of a compact metric space $(X,d)$ without isolated points.

%%%%%%%%%%%%%%%%%%%%%%%%%%%%%%%%%%%%%%%%%%%%%%%%%%%%%%%%%%%

\noindent $\textbf{(a)}$ Suppose that $A$ is dense in $X$ and $T$ has the barycenter property in $A$. Let $U$ and $V$ be nonempty open subsets of $X$. Since $A$ is dense in $X$, there are $p \in U$ and $q \in V$ such that $p,q \in A$; moreover, there exists $\varepsilon>0$ such that $B(p,\varepsilon) \subseteq U$ and $B(q,\varepsilon) \subseteq V$. By the barycenter property, there is a positive integer $N = N(\varepsilon, p, q)$ such that, for  $n_1 =1 =n_2$, one can find an integer $m \in [0, N]$ and $x_0 \in X$ satisfying $d\big(T^i(x_0), T^i(p)\big) < \varepsilon$ for every $-n_1 \leqslant  i \leqslant 0$ and $d\big(T^{i+m}(x_0), T^i(q)\big) < \varepsilon$ for every $0 \leqslant  i \leqslant n_2$. In particular, $d\big(x_0, p\big) < \varepsilon$  and $d\big(T^{m}(x_0), q\big) < \varepsilon$. So, $T^{m} (U) \cap V \neq \emptyset$. Hence $T$ is topologically transitive.

%%%%%%%%%%%%%%%%%%%%%%%%%%%%%%%%%%%%%%%%%%%%%%%%%%%%%%%%%%%%%

\noindent $\textbf{(b)}$  Suppose that $\Omega(T) = CR(T)$ and that there is a subset $A$ of $X$ such that $T(A) = A$, $\Omega(T) = \overline{A}$ and $T$ has the barycenter property in $A$. We start by showing that, under these assumptions, $T$ is chain-transitive.

\begin{lemma}\label{lemma-fundamental}
Let $T\colon X \to X$ be a homeomorphism of a compact metric space $(X,d)$. If there is a subset $A$ of $X$ such that $T(A)=A$, $\Omega(T) =  \overline{A}$ and $T$ has the barycenter property in $A$, then $T$ is chain-transitive.
\end{lemma}

\begin{proof} Take $x,y \in X$ and fix $\varepsilon>0$. We are looking for an $\varepsilon-$chain from $x$ to $y$, that is, a finite set of points in $X$, say $x_0 = x,\, x_{1}, \cdots, x_{N} = y$, such that 
$$
d(T(x_{i}), x_{i+1}) \,<\, \varepsilon \quad \quad \forall \, i \in \{0,1,2,\cdots, N-1\}.
$$
Given $a \in \omega(x) \subseteq  \Omega(T)$  and $b \in  \alpha(y)\subseteq  \Omega(T)$, there exist $k, \ell \in \mathbb{N}$ such that 
\begin{equation}\label{eq:kell}
d(T^{k}(x), a) \,<\,\frac{\varepsilon}{2} \quad \quad \text{and} \quad \quad  d(T^{-\ell}(y), b) \,<\,\frac{\varepsilon}{2}.
\end{equation}
Since $\Omega(T) = \overline{A}$ and $a,b \in \Omega(T)$, there are $p, \,q \in A$ such that 
\begin{equation}\label{eq:close}
d(p,a)\, \,< \frac{\varepsilon}{2} \quad \quad \text{and} \quad \quad  d(q,b) \,<\, \frac{\varepsilon}{2}.
\end{equation} 
From \eqref{eq:close} and \eqref{eq:kell}, we deduce that
\begin{eqnarray}\label{eq:dist}
d(T^{k}(x),p) &\leqslant& d(T^{k}(x),a)+ d(a,p) \,<\,\frac{\varepsilon}{2} + \frac{\varepsilon}{2} \,=\, \varepsilon \nonumber \\
d(T^{-\ell}(y), q) &\leqslant & d(T^{-\ell}(y), b)  + d(q, b) \,< \,\frac{\varepsilon}{2} + \frac{\varepsilon}{2} \,=\, \varepsilon. 
\end{eqnarray}
Summoning the barycenter property for $T(p)$ and $T^{-1}(q)$, which is possible because $T(A) = A$, we may take a positive integer $N= N(\varepsilon, T(p), T^{-1}(q))$ such that, for any $n_1, n_2 \in \mathbb{N}$, one can find an integer $m \in [0, N]$ and $z \in X$ satisfying 
\begin{eqnarray*}
d\big(T^i(z), T^i(T(p))\big) &<& \varepsilon \quad \quad \forall\, -n_1 \,\leqslant\,  i \,\leqslant\, 0 \\
d\big(T^{i+m}(z), T^i(T^{-1}(q))\big) &<& \varepsilon \quad \quad \forall \,\, 0\, \leqslant\,  i\, \leqslant\, n_2.
\end{eqnarray*}
In particular, 
\begin{equation}\label{eq:zpq}
d\big(z, T(p) \big) \,<\, \varepsilon \quad \quad \text{and} \quad \quad d\big(T^{m}(z), T^{-1}(q)\big) \,<\, \varepsilon
\end{equation}
Define
\begin{align*}
x_{0}&= x\\
x_{1}&= T(x) & x_{2}&= T^2(x) & & \cdots & x_{k-1}&=  T^{k-1}(x)\\ 
x_{k}&= p & x_{k+1}&= z  & &\cdots & x_{k+m}&= T^{m-1} (z)\\
x_{k+m+1}&= T^{-1}(q) & x_{k+m+2}&= T^{-\ell}(y) & &\cdots & x_{k +m+2 +(\ell-1)}&= T^{-1}(y)\\
x_{k +m+2 +\ell}&= y.
\end{align*}
\smallskip

\noindent From the inequalities \eqref{eq:dist} and \eqref{eq:zpq}, we conclude that the points $(x_{i})_{i=1}^{k +m+2 +\ell}$ form an $\varepsilon-$chain linking $x$ to $y$. So, $T$ is chain-transitive.
\end{proof}

By Lemma~\ref{lemma-fundamental}, one has that $T$ is chain-transitive, so $CR(T)=X$. Thus, $\Omega(T) = X$ since, by assumption, $\Omega(T) = CR(T)$. From $\Omega(T) = \overline{A}$ and $\Omega(T) = X$, we obtain that $A$ is dense in $X$. So the hypothesis on item (a) of Theorem \ref{maintheo-weak-hmp} are fulfilled, and therefore $T$ is topologically transitive.

\noindent $\textbf{(c)}$ Suppose that $T$ is topologically transitive and has the shadowing property. The following auxiliary result provides the maximum length of a chain between any two elements of $X$  
when they are subject to a topologically transitive dynamics.

\begin{lemma}\label{lemma:transitive-chain}
\emph{Let $T\colon X \to X$ be a homeomorphism of a compact metric space $(X,d)$ without isolated points. If $T$ is topologically transitive, then for every $\xi>0$ there exists a positive integer $N =N(\xi)$ such that for any $a,b \in X$ there is a $\xi-$chain from $a$ to $b$ with  length $L$, where $2 \leq  L  \leq N$.}
\end{lemma}

\begin{proof}
Recall that one-sided topological transitivity is equivalent to topological transitivity whenever the compact metric space does not contain isolated points (cf. \cite[Theorem~1.4]{AC12}). So, $T$ is one-sided topologically transitive. Consider an open finite subcover of $X$, given by $\{U_{1}\cdots, U_{s}\}$, such that $\diam(U_{i}) < \xi $ for every $i \in \{1,\cdots, s\}$. Thus, by the one-sided transitivity, for each pair $(i,j) \in \{1,\cdots, s\}\times \{1,\cdots, s\}$ there exists $t_{(i,j)} \in \mathbb{N}$ such that $T^{-t_{(i,j)}}(U_{i}) \cap U_{j} \neq \emptyset$. Define 
$$N \,=\, N(\xi) \,=\, 2+\max \big\{ t_{(i,j)}\colon \, i,j \in \{1,\cdots, s \}\big\}.$$
We claim that for any $a,b \in X$  there is a $\xi-$chain from $a$ to $b$ of length $p$ with $2 \leq  p  \leq N$. Indeed, since $X = \bigcup_{i=1}^{s} U_{i}$, there are $r, m$ such that $T(a) \in U_{r}$ and $b \in U_{m}$. Using that $T^{-t_{(m,r)}}(U_{m}) \cap U_{r} \neq \emptyset$, take $c \in U_{r}$ such that $T^{t_{(m,r)}} (c) \in U_{m}$. Then, define a $\xi-$pseudo-orbit as follows: 
$$x_{0} \,=\, a,\quad  x_{1} \,=\, c,\quad    x_{2} \,=\, T(c), \, \cdots, \, x_{t_{(m,r)}} \,=\, T^{t_{(m,r)} -1} (c), \quad  x_{t_{(m,r)} \,+\,1} \,=\,b$$
Note that the length of this chain is $L = t_{(m,r)} + 2 $; hence $2 \leq L  \leq N$.
\end{proof}

Let $A$ be a nonempty subset of $X$. Take $p, \, q \in A$ and fix $\varepsilon>0$.  Let $\delta > 0$, provided by the shadowing property, be such that every $\delta-$pseudo-orbit of $T$ is $\varepsilon-$shadowed in $X$. Consider $U = B(T(p), \delta)$ and $V = B(q,\delta)$, where $B(a,r)$ stands for the open ball in $(X,d)$ centered at $a$ with radius $r$.  From \cite[Theorem~1.4]{AC12}, we know that one-sided topological transitivity is equivalent to topological transitivity whenever the compact metric space does not contain isolated points. So, $T$ is one-sided topologically transitive.  Applying Lemma~\ref{lemma:transitive-chain} when $\xi = \delta$, we find a positive integer $N =N(\delta)$ such that there is a  $\delta-$chain  $(y_{j})_{1}^{m}$ from $T(p)$ to $q$ with length $m$, where $2 \leq  m  \leq N$, satisfying
$$y_{1} \,=\,T(p)  \quad \quad \text{and} \quad \quad y_{m}\,=\ q.$$ 
Define a $\delta-$pseudo-orbit as follows:
\begin{eqnarray*}
x_{i} &=& T^{i}(p) \quad \quad \forall \, i \leqslant 0\\
x_{1} &=& y_{1}, \quad x_{2} \,=\, y_2, \quad \cdots, \quad x_{m-1} \,=\, y_{m-1}\\
x_{m+ \ell} &=& T^{\ell}(q) \quad \quad \forall\, \ell \geqslant  0.
\end{eqnarray*}
By the shadowing property, there exists $z \in X$ such that $d(T^{k}(z), x_{k})< \varepsilon$ for all $k \in \mathbb{Z}$. In particular, 
\begin{eqnarray*}
d(T^{i}(z), \,x_{i}) &=& d(T^{i}(z),\,T^{i}(p)) \,<\, \varepsilon \quad \quad \forall\, i \leqslant 0 \\
d(T^{m+ \ell}(z), \,x_{m+ \ell}) &=&  d(T^{m+ \ell}(z), \,T^{\ell}(q)) \, < \, \varepsilon \quad \quad \forall\, \ell  \geqslant 0.
\end{eqnarray*}
Therefore, $T$ has the barycenter property in $A$. \qed

%%%%%%%%%%%%%%%%%%%%%%%%%%%%%%%%%%%%%%%%%%%%%%%%%%%%%%%%%%%%%

\noindent $\textbf{(d)}$ Suppose that $T$ has the shadowing property. Thus, $\Omega(T) = CR(T)$ (cf. \cite[Theorem~3.1.2]{AH94}). Therefore, if there exists $A \subseteq X$ such that $T(A) = A$, $\Omega(T) = \overline{A}$ and $T$ has the barycenter property in $A$, then by item (b) we know that $T$ is topologically transitive. \qed

%%%%%%%%%%%%%%%%%%%%%%%%%%%%%%%%%%%%%%%%%%%%%%%%%%%%%%%%%%%%%

\noindent $\textbf{(e)}$ Assume that $T$ has the shadowing property and is accessible in $R(T)$. We start by proving the following auxiliar result which replaces, in this context, Lemma~\ref{lemma-fundamental}.

\begin{lemma}\label{accessible-lemma}
Let $(X,d)$ be a compact metric space and $T\colon X \to X$ be a homeomorphism. If $\Omega(T) \subseteq \overline{R(T)}$ and $T$ is accessible in $R(T)$, then $T$ is chain-transitive.
\end{lemma}

\begin{proof}
Take $x,y \in X$ and fix $\varepsilon>0$. We are looking for an $\varepsilon-$chain from $x$ to $y$, that is, a finite set of points in $X$, say $x = x_{0},\, x_{1}, \cdots, x_{N} = y$, such that 
$$
d(T(x_{i}), x_{i+1}) \,<\, \varepsilon \quad \quad \forall \, i \in \{0,1,2,\cdots, N-1\}.
$$
Given $a \in \omega(x) \subseteq  \Omega(T)$  and $b \in  \alpha(y)\subseteq  \Omega(T)$, there exist $k, \ell \in \mathbb{N}$ such that 
\begin{equation}\label{eq:kell-1}
d(T^{k}(x), a) \,<\,\frac{\varepsilon}{2} \quad \quad \text{and} \quad \quad  d(T^{-\ell}(y), b) \,<\,\frac{\varepsilon}{2}.
\end{equation}
Since $\Omega(T) \subseteq \overline{R(T)}$ and $a,b \in \Omega(f)$, there are $p, \,q \in R(T)$ such that 
\begin{equation}\label{eq:close-1}
d(p,a)\, \,< \frac{\varepsilon}{2} \quad \quad \text{and} \quad \quad  d(q,b) \,<\, \frac{\varepsilon}{2}.
\end{equation} 
Thus,
\begin{eqnarray}\label{eq:pq}
d(T^{k}(x),p) &\leqslant& d(T^{k}(x),a)+ d(a,p) \,<\,\frac{\varepsilon}{2} + \frac{\varepsilon}{2} \,=\, \varepsilon \nonumber\\
d(T^{-\ell}(y), q) &\leqslant & d(T^{-\ell}(y), b)  + d(q, b) \,< \,\frac{\varepsilon}{2} + \frac{\varepsilon}{2} \,=\, \varepsilon. 
\end{eqnarray}

\noindent The next two assertions are straightforward:

\noindent \textbf{Claim 1}: For every $\xi_1, \xi_2 \in X$, if $\xi \in  \alpha(\xi_1) \cap \omega(\xi_2)$ then $T^{r}(\xi) \in  \alpha(T^{s}(\xi_1)) \cap \omega(T^{s}(\xi_2))$ for all $r, s \in \mathbb{Z}$. 

\noindent \textbf{Claim 2}: If $\gamma \in R(T)$, then $T^{-1}(\gamma) \in R(T)$. 

So, as $q \in R(T)$, then $T^{-1}(q) \in R(T)$. 
Since $T$ is accessible in $R(T)$, there exists a $su-$path in $R(T)$ from $p$ to $T^{-1}(q)$, namely a finite set of points  $z_{0}, \cdots, z_{m} \in R(T)$ such that $z_{0}=p$, $z_{m}=T^{-1}(q)$ and $z_{i} \in W^{s}(z_{i-1}) \cup W^{u}(z_{i-1})$ for $ i \in \{1,\cdots,m\}.$

\noindent \textbf{Claim 3}: $\omega(z_{i-1}) \cap \alpha(z_{i}) \neq \emptyset$ for every $i \in \{1,\cdots,m\}.$

\begin{proof}
Fix $i \in \{1,\cdots,m\}$. One has $z_{i} \in W^{s}(z_{i-1}) \cup W^{u}(z_{i-1})$. If $z_{i} \in W^{s}(z_{i-1})$,  then $\omega(z_{i}) = \omega(z_{i-1})$. Moreover, as $ z_{i} \in R(T)$, then $\alpha(z_{i}) \cap \omega(z_{i}) \neq \emptyset $. Hence, $\omega(z_{i-1}) \cap \alpha(z_{i}) \neq \emptyset$.  
If, otherwise, $z_{i} \in W^{u}(z_{i-1})$, then $\alpha(z_{i}) = \alpha(z_{i-1})$. Since $ z_{i-1} \in R(T)$, we conclude that $\alpha(z_{i-1}) \cap \omega(z_{i-1}) \neq \emptyset $. So $\omega(z_{i-1}) \cap \alpha(z_{i}) \neq \emptyset$. \end{proof}

Let us now construct an $\varepsilon-$chain joining $x$ to $y$. By Claim 3, given $ i \in \{1,\cdots,m\}$ there exists $v_{i} \in \omega(z_{i-1}) \cap \alpha(z_{i})$. By Claim 1, $T(v_{i})  \in \omega(z_{i-1}) \cap \alpha(z_{i})$ as well. Therefore, for every $ i \in \{1,\cdots,m\}$, there are $n_{i},m_{i} \in \mathbb{N}$ such that 
\begin{equation}\label{eq:zv}
d(T^{n_{i}}(z_{i-1}), \,v_{i} ) \,<\, \varepsilon \quad \quad \text{and} \quad \quad d(T^{-m_{i}}(z_{i}), \,T(v_{i}) ) \,<\, \varepsilon.
\end{equation}
Define the following collection of points in $X$:
\begin{align*}
x_0 &= x, \,\, x_{1} = T(x), \,\,\cdots,\,\, x_{k-1} =  T^{k-1} (x) \\
x_{k} &= p = z_{0}, \,\, x_{k+1} = T(z_{0}), \,\,\cdots, \,\, x_{k + n_{1} - 1} =  T^{n_{1} - 1}(z_{0})\\
x_{k+n_{1} } &= v_{1}\\
x_{k + n_{1} +1} &= T^{-m_{1}}(z_{1}), \,\, \cdots, \,\ x_{k +n_{1} +1 + m_{1}} = z_{1}\\
x_{k +n_{1} +1 + m_{1}+1} &= T(z_{1}), \,\, \cdots, \,\, x_{\gamma +n_{1} +1 + m_{1}+n_{2} -1} = T^{n_{2}-1}(z_{1})\\
x_{k +n_{1} +1 + m_{1}+n_{2}} &= v_{2}
\end{align*}
\begin{align*}
x_{k+n_{1} +1 + m_{1}+n_{2} +1 } &= T^{-m_{2}}(z_{2}), \,\,\cdots, \,\, x_{k +n_{1} +1 + m_{1}+n_{2} +1 + m_{2} } = x_{k + \sum_{i=1}^{2 } (n_{i} +1 + m_{i}) } = z_{2}\\
x_{k + j+ \sum_{i=1}^{2 } (n_{i} +1 + m_{i}) } &= T^{j}(z_{2})\quad \quad  \forall\, j \in \{0,\cdots , n_{3} -1 \}\\
x_{k  + n_{3}  + \sum_{i=1}^{2 } (n_{i} +1 + m_{i}) } &= v_{3}\\
x_{k + n_{3} +1+j + \sum_{i=1}^{2 } (n_{i} +1 + m_{i}) } &= T^{-m_{3} +j}(z_{3}) \quad \quad \forall\, j \in \{0,\cdots , m_{3}  \}\\
x_{k + \sum_{i=1}^{3 } (n_{i} +1 + m_{i}) } &= z_{3}\\
&\,\,\vdots\\
x_{k + \sum_{i=1}^{m } (n_{i} +1 + m_{i}) } &= z_{m} = T^{-1}(q)\\
x_{k +1+j+ \sum_{i=1}^{m } (n_{i} +1 + m_{i}) } & = T^{-\ell + j}(y) \quad \quad  \forall\, j \in \{0,\cdots , \ell\}.
\end{align*}

\noindent Due to the inequalities \eqref{eq:close-1},  \eqref{eq:pq} and \eqref{eq:zv}, the points 
$(x_{i})_{i=0}^{k+1+\ell+ \sum_{i=1}^{m} (n_{i} +1 + m_{i})}$ build an $\varepsilon-$chain linking $x$ to $y$. So, $T$ is chain-transitive.
\end{proof}

We may now finish the proof of Theorem~\ref{maintheo-weak-hmp}(e). From \cite[Corollary~2.3]{CK96}, since $T$ has the shadowing property we know that $\overline{R(T)} = \Omega(T)$. On the other hand, by Lemma \ref{accessible-lemma}, $T$ is chain-transitive. As, by assumption, $T$ also has the shadowing property, then $T$ is one-sided topologically transitive due to \cite[Lemma~2.1]{L12}. Thus, $T$ is topologically transitive. \qed

%%%%%%%%%%%%%%%%%%%%%%%%%%%%%%%%%%%%%%%%%%%%%%%%%%%%%%%%%%%

\noindent $\textbf{(f)}$ Assume that $T$ has the shadowing property and the $su-$intersecting property in $X$. By Lemma~\ref{lemma-iterations-hyperbolic-meet}, we know that, for every positive integer $k$, the homeomorphism $T^{k}$ has the $su-$intersecting property in $X$. Moreover, according to Lemma~\ref{lemma-chain-transitive}, $T^{k}$ is chain-transitive for every positive integer $k$. Besides, as $T$ has the shadowing property then $T^{k}$ also has the shadowing property for every positive integer $k$ (cf. \cite[Theorem~2.3.3]{AH94}). Since the shadowing property together with the chain-transitivity imply topological transitivity, we conclude that $T^{k}$ is topologically transitive for every positive integer $k$. Finally, since $T$ has the shadowing property and $T^{k}$ is topologically transitive for every positive integer $k$, then $T$ is topologically mixing by \cite[Theorem~45]{KLO16}. \qed

%%%%%%%%%%%%%%%%%%%%%%%%%%%%%%%%%%%%%%%%%%%%%%%%%%%%%%%%%%

\section{Proof of Theorem~\ref{maintheo-Homeo-hyperbolicmeet}}\label{se:proof-thmB}

Theorem~\ref{maintheo-Homeo-hyperbolicmeet} is a particular instance of more general results that we shall prove in this section. The next one comprehensibly contains item (a) of Theorem~\ref{maintheo-Homeo-hyperbolicmeet}.

\begin{proposition}\label{theo-Homeo-hyperbolicmeet}
Let $(X,d)$ be a compact metric space without isolated points and $T\colon X \to X$ be an expansive homeomorphism which satisfies the shadowing property. Then the following assertions are equivalent:
\begin{itemize}
\item[$(i)$] $T$ has the two-sided limit shadowing property.

\item[$(2i)$] $T$ has the $su-$intersecting property in $X$.

\item[$(3i)$] $T$ has the $su-$intersecting property in $\per(T)$.

\item[$(4i)$] $T$ has the specification property.
\end{itemize}
\end{proposition}

\begin{proof}

We will show that  $(i) \Rightarrow (2i) \Rightarrow (3i) \Rightarrow (4i) \Rightarrow (i)$.
\medskip

\noindent $(i) \Rightarrow (2i)$  This is \cite[Lemma~2.4]{C15}.
\medskip

\noindent $(2i) \Rightarrow (3i)$ This is clear.
\medskip

\noindent $(3i) \Rightarrow (4i)$ Since $T$ is expansive and has the shadowing property, one has $\overline{\per (T)} = \Omega(T)$ according to \cite[Theorem~3.1.8]{AH94}.

\begin{lemma}\label{lemma-fundamental-corollary}
Let $(X,d)$ be a compact metric space, $T \colon X\to X$ be homeomorphism. If $\Omega(T) = \overline{\per(T)}$ and $T$ has the $su-$intersecting property in $\per(T)$, then $T^{k}$ is chain-transitive for every positive integer $k$.
\end{lemma}

\begin{proof}

The case $k=1$ was established in Lemma~\ref{lemma-fundamental}. Fix a positive integer $k>1$. Since $T$ has the $su-$intersecting property in $\per(T)$, by Lemma~\ref{lemma-iterations-hyperbolic-meet} we know that $T^{k}$ has the $su-$intersecting property in $\per(T)$. Moreover, $\per(T^{k}) = \per(T)$, hence $T^{k}$ also has the $su-$intersecting property in $\per(T^k)$. Since 
$$\overline{\per(T)} \,=\, \overline{\per(T^{k})} \,\subseteq\, \Omega(T^{k}) \,\subseteq\, \Omega(T) \,=\, \overline{\per(T)}$$
we get $\overline{\per(T^{k})} = \Omega(T^{k})$. Thus, we may apply Lemma~\ref{lemma-iterations-hyperbolic-meet} to $T^{k}$ and this way conclude that $T^{k}$ is chain-transitive. 
\end{proof}

Let us resume the proof of the implication $(3i) \Rightarrow (4i)$. Suppose that the homeomorphism $T$ has the $su-$intersecting property in $\per (T)$. By Lemma~\ref{lemma-fundamental-corollary}, $T^{k}$ is chain-transitive for every positive integer $k$. Moreover, as $T$ has the shadowing property, $T^k$ has this property as well for every positive integer $k$ (cf. \cite[Theorem~2.3.3]{AH94}). In addition, a chain-transitive homeomorphism with the shadowing property is topologically transitive. So, $T^k$ is topologically transitive for every positive integer $k$. Using that one-sided topological transitivity is equivalent to topological transitivity whenever the compact metric space does not contain isolated points (cf. \cite[Theorem~1.4]{AC12}), we conclude that $T^k$ is one-sided topologically transitive for every positive integer $k$. Therefore, $T$ has the specification property by \cite[Theorem~45]{KLO16}.

\noindent $(4i) \Rightarrow (i)$  This is \cite[Lemma~2.2]{C15}.
\end{proof}

We draw item (b) of Theorem~\ref{maintheo-Homeo-hyperbolicmeet} from the next result.

\begin{proposition} Let $(X,d)$ be a compact connected metric space 
and $T\colon X \to X$ be an expansive homeomorphism which satisfies the shadowing property. Then the following conditions are equivalent:
\begin{itemize}
\item[$(i)$] $T$ has the $su-$intersecting property in $X$.

\item[$(2i)$] $T$ has the $su-$intersecting property in $\per(T)$.

\item[$(3i)$] $T$ has the weak $su-$intersecting property in $\per(T)$.

\item[$(4i)$] $T$ has the barycenter property in $\per(T)$.
\end{itemize}
\end{proposition}

\begin{proof}
We reason this way: $(i) \Rightarrow (2i) \Rightarrow (3i) \Rightarrow (4i) \Rightarrow (i)$.

\noindent $(i) \Rightarrow (2i)$  This is clear
\medskip

\noindent $(2i) \Rightarrow (3i)$ This is clear.
\medskip

\noindent $(3i) \Rightarrow (4i)$ This is \cite[Lemma~2.2]{ST12}.
\medskip

\noindent $(4i) \Rightarrow (i)$ Since $T$ is expansive and has the shadowing property, then $\overline{\per (T)} = \Omega(T)$ by  \cite[Theorem~3.1.8]{AH94}. So, if the homeomorphism $T$ has the barycenter property in $\per(T)$ then, by Theorem \ref{maintheo-weak-hmp}, $T$ is topologically transitive. Since $T$ is topologically transitive, has the shadowing property and $X$ is connected, then $T$ is topologically mixing (cf. \cite[Proposition~3]{C15TS}). Therefore, $T$ has the specification property (cf. \cite[Theorem~45]{KLO16}). Hence, by Proposition~\ref{theo-Homeo-hyperbolicmeet}, $T$ has the $su-$intersecting property in $X$.
\end{proof}

%%%%%%%%%%%%%%%%%%%%%%%%%%%%%%%%%%%%%%%%%%%%%%%%%%%%%%%%%%%%%%%%%%%%%%%%%%%%%%%%%%%%%%%%%%%%%

\section{Proof of Theorem~\ref{theo-aux-complet}}\label{se:proof-thmC}

We will use the following scheme of implications: 

\footnotesize
\begin{tabular}{lllllllllllllllllllllll}
&  &  &  &  & & & & & & $(13)$ & $\Rightarrow$ & $(4)$ &  &  \\
&  &  &  &  & & & & & & $\,\,\,\Uparrow$ &  &  &  &  \\
$(1)$ & $\Rightarrow$ & $(2)$ & $\Rightarrow$ & $(3)$ & $\Rightarrow$ & $(4)$ & $\Rightarrow$ & $(5)$ & $\Rightarrow$ & 
$\,(6)$ & $\Rightarrow$ & $(7)$ & $\Rightarrow$ & $(8)$ 
& $\Rightarrow$ & $(9)$ & $\Rightarrow$ & $(10)$ &&   \\
&&&&&&&&&& $\,\,\,\Downarrow$ &&&&&&&& $\,\,\,\Downarrow$ &&&&\\
&&&&&&&&&& $(12)$ & $\Rightarrow$ & $(11)$ & $\Rightarrow$ & $(1)$ &&&& $(11)$&&&&
\end{tabular}
\normalsize

\medskip

\noindent $(1) \Rightarrow (2)$ This is \cite[Lemma~2.4]{C15}.

\noindent $(2) \Rightarrow (3)$ This is clear.

\noindent $(3) \Rightarrow (4)$ This is also immediate since $W^{s}_f(x) \cap W^{u}_f(y) \subseteq W^{s}_f \big(\mathcal{O}_f(x)\big) \cap W^{u}_f \big(\mathcal{O}_f(y)\big)$.
 
\noindent $(4) \Rightarrow (5)$  Suppose that $f$ has the weak $su-$intersecting property in $\per(f)$. According to \cite[Lemma~2.2]{ST12},  $f$ has the barycenter property in $\per(f)$. Since $f$ is Axiom $A$, one has $\Omega(f) \,=\, \overline{\per(f)} $.
   
\noindent $(5) \Rightarrow (6)$  As $f$ is $\Omega-$stable,  
by the $\Omega-$stability theorem \cite{Smale70} the chain-recurrent set of $f$, say $CR(f)$, is hyperbolic. Moreover, since $M$ has dimension $\geqslant 2$, by the Closing Lemma \cite{Pugh} one has 
$$\Omega(f) \,=\, \overline{\per(f)} \,=\, CR(f).$$
As, in addition, $f$ has the barycenter property in $A$, $f(A)=A$ and $\overline{A} = \Omega(f)$, we already know from Lemma~\ref{lemma-fundamental} that $f$ is chain-transitive. So $CR(f) = M$. Consequently, $\Omega(f) = M$.
Since $\Omega(f)$ is hyperbolic, we conclude that $M$ is hyperbolic, that is, $f$ is an Anosov diffeomorphism. Thus, $f$ has the shadowing property. Besides, as already mentioned, a chain-transitive homeomorphism with the shadowing property is topologically transitive. Hence $f$ is a topologically transitive Anosov diffeomorphism. 

\noindent $(6) \Rightarrow (7)$ Suppose that $f$ is a topologically transitive Anosov diffeomorphism. So $f$ is an expansive topologically transitive diffeomorphism that has the shadowing property. Since $M$ is connected, then $f$ is topologically mixing \cite[Proposition~3]{C15TS}. Therefore, by \cite[Theorem~45]{KLO16}, $f$ has the specification property.

\noindent $(7) \Rightarrow (8)$ This is a consequence of  \cite[Proposition~2.1]{PfSu07}. See also the remarks on \cite[page~4]{Thomp} and \cite[page~249]{KPR14}.

\noindent $(8) \Rightarrow (9)$ This is \cite[Theorem~3.5]{KPR14}.

\noindent $(9) \Rightarrow (10)$ This is \cite[Theorem~3.7]{KPR14}.

\noindent $(10) \Rightarrow (11)$ This is \cite[Lemma~3.1]{KPR14}.

\noindent $(11) \Rightarrow (1)$ Since $f$ is $\Omega-$stable and, in addition, chain-transitive, we conclude as previously  that $\Omega(f) = \overline{\per(f)} = CR(f) = M$ and $\Omega(f)$ is hyperbolic. So $f$ is a chain-transitive diffeomorphism satisfying the shadowing property, and therefore topologically transitive. By \cite[Corollary~2.3]{C15}, the topologically transitive Anosov diffeomorphism $f$ has the two-sided limit shadowing property.

\noindent $(6) \Rightarrow (12)$ Suppose that $f$ is a topologically transitive Anosov diffeomorphism. So $f$ has the  shadowing property. Then, by \cite[Theorem~A]{LO18}), one has $\overline{\mathcal{P}_{f}^{erg}(M)} = \mathcal{P}_{f}(M)$.

\noindent $(12) \Rightarrow (11)$ Suppose now that $f\colon M \to M$ is an $\Omega-$stable diffeomorphism such that $\mathcal{P}_{f}^{erg}(M)$ is dense in $\mathcal{P}_{f}(M)$. 

As explained before, by the $\Omega-$stability of $f$, one has $\Omega(f) = \overline{\per(f)} = CR(f)$ and $CR(f)$ is hyperbolic.  Besides, by definition, every periodic point of $f$ is an element of the measure center $\mathcal{M}(M,f)$. Moreover, due to  \cite[Lemma~5.2]{CCS24}, one has 
$$\mathcal{M}(M,f) \,\subseteq\, \Omega(f)$$
and so 
$$\Omega(f) \,=\, \overline{\per(f)} \,=\, CR(f) \,=\, \mathcal{M}(M,f).$$
Now, with the equalities $\overline{\mathcal{P}_{f}^{erg}(M)} = \mathcal{P}_{f}(M)$ and  $\mathcal{M}(M,f) = CR(f)$, we deduce from \cite[Lemma~5.8]{CCS24} that $f$ is topologically transitive. In particular, $f$ is chain-transitive.

\noindent $(6) \Rightarrow(13)$ Suppose that $f$ is a topologically transitive Anosov diffeomorphism. Then $f$ has the shadowing property. Therefore, $f$ has the gluing orbit property, as shown by the next lemma.

\begin{lemma}\label{lemma:sh+tr+connct-gluing}
Let $T\colon X \to X$ be a homeomorphism of a compact metric space $(X,d)$ without isolated points. If $T$ has the shadowing property and is topologically transitive, then $T$ has the gluing orbit property.
\end{lemma}

\begin{proof}
Fix $\varepsilon>0$.  Let $\delta > 0$ be the value provided by the shadowing property, so that every $\delta-$pseudo-orbit of $T$ is $\varepsilon-$shadowed in $X$. Take $k \in \mathbb{N}$, $y_1, y_2, \dots, y_k \in X$ and positive integers $n_1, \dots, n_k$.

Applying Lemma~\ref{lemma:transitive-chain} when $\xi = \delta$, we deduce that there exists $N = N(\delta)$ such that there are:
\begin{itemize}
\item a $\delta-$chain $(x_{1,2})_{0}^{N-1}$ from  $T^{n_{1}}(y_{1})$ to $y_{2}$ with length $p_{1}$ where  $2 \leq p_{1}\leq N$ and satisfying
$$(x_{1,2})_{0} \,=\,T^{n_{1}}(y_{1}) \quad \quad \text{and} \quad \quad (x_{1,2})_{p_{1}-1}\,=\,y_{2};$$
\item a $\delta-$chain $(x_{2,3})_{0}^{N-1}$ from  $T^{n_{2}}(y_{2})$ to $y_{3}$ with length $p_{2}$ where  $2 \leq p_{2}\leq N$ and 
$$(x_{2,3})_{0} \,=\,T^{n_{2}}(y_{2})  \quad \quad \text{and} \quad \quad  (x_{2,3})_{p_{2}-1}\,=\,y_{3};$$
\item a $\delta-$chain $(x_{3,4})_{0}^{N-1}$ from  $T^{n_{3}}(y_{3})$ to $y_{4}$ with length $p_{3}$ where  $2 \leq p_{3}\leq N$ with
$$(x_{3,4})_{0}\,=\,T^{n_{3}}(y_{3}) \quad \quad \text{and} \quad \quad (x_{3,4})_{p_{3}-1}\,=\,y_{4};$$
$\vdots$
\item a $\delta-$chain $(x_{k-1,k})_{0}^{N-1}$ from  $T^{n_{k-1}}(y_{k-1})$ to $y_{k}$ with length $p_{(k-1)}$, where
$$(x_{k-1,k})_{0}\,=\,T^{n_{k-1}}(y_{k-1}) \quad \quad \text{and} \quad \quad (x_{k-1,k})_{p_{(k-1)}-1}\,=\,y_{k}$$
and $2 \leq p_{(k-1)}\leq N$.
\end{itemize}
Define a $\delta-$pseudo-orbit as follows:
\begin{eqnarray*}
x_{\ell} &=& T^{\ell}(y_{1}) \quad \forall\, \ell < 0 \\
x_{0} &=& y_{1}, x_{1} \,=\, T^{1} (y_{1}), \cdots, x_{n_{1}} \,=\, T^{n_{1}}(y_{1}) \,=\, (x_{1,2})_{0}\\
x_{n_{1}+1} &=& (x_{1,2})_{1}, \cdots, x_{n_{1}+(p_{1}-1)} \,=\, (x_{1,2})_{p_{1}-1} \,=\, y_{2}\\
x_{n_{1}+(p_{1}-1) + 1} &=& T(y_{2}), \cdots,  x_{n_{1}+(p_{1}-1)+n_{2}} \,=\, T^{n_{2}}(y_{2}) \,=\,  (x_{2,3})_{0}\\
x_{n_{1}+(p_{1}-1)+n_{2}+1} &=&  (x_{2,3})_{1},\cdots, x_{n_{1}+(p_{1}-1)+n_{2} +(p_{2}-1)} \,=\,  (x_{2,3})_{p_{2}-1}\,=\,y_{3}\\
x_{n_{1}+(p_{1}-1)+n_{2} +(p_{2}-1) +1} &=& T(y_{3}),\cdots, x_{n_{1}+(p_{1}-1)+n_{2} +(p_{2}-1) +n_{3}} \,=\,T^{n_{3}}(y_{3}) \,=\, (x_{3,4})_{0}\\
x_{n_{1}+(p_{1}-1)+n_{2} +(p_{2}-1) +n_{3}+1} &=& (x_{3,4})_{1},\cdots, x_{n_{1}+(p_{1}-1)+n_{2} +(p_{2}-1) +n_{3}+(p_{3}-1)} \\
&=& (x_{3,4})_{p_{3}-1} \,=\,y_{4} \\
&\vdots & \\
x_{\left( 1+ \sum_{i=1}^{k-2} (n_{i} + (p_{i}-1)) \right)} &=& T(y_{k-1}),\cdots, x_{\left( n_{(k-1)}+ \sum_{i=1}^{k-2} (n_{i} + (p_{i}-1))\right) } \,=\,T^{n_{(k-1)}}(y_{k-1}) \\
&=& (x_{k-1,k})_{0}\\
x_{\left( 1+  n_{(k-1)}+ \sum_{i=1}^{k-2} (n_{i} + (p_{i}-1))\right)} &=& (x_{k-1,k})_{1},\cdots, x_{\left( (p_{(k-1)}-1)+  n_{(k-1)}+ \sum_{i=1}^{k-2} (n_{i} + (p_{i}-1))\right)} \\
&=& (x_{k-1,k})_{p_{(k-1)}-1} \,=\, y_{k}\\
x_{\left(  \ell +  \sum_{i=1}^{k-1} (n_{i} + (p_{i}-1))\right)} &=& T^{\ell}(y_{k}) \quad \forall \ell \geq 0.
\end{eqnarray*}
By the shadowing property, there exists $z \in X$ such that $d(T^{j}(z), x_{j})< \varepsilon$ for all $j \in \mathbb{Z}$. Hence, $T$ has the gluing orbit property.
\end{proof}

\noindent $(13)\Rightarrow (4)$  Suppose that $f$ has the gluing orbit property.

\begin{lemma}\label{lemma:gluing-bary}
Let $T\colon X \to X$ be a homeomorphism of a compact metric space $(X,d)$ without isolated points and $A$ be a nonempty subset of $X$. If $T$ has the gluing orbit property, then $T$ has the barycenter property in $A$.
\end{lemma}

\begin{proof}
Consider $p, q \in A$ and $\varepsilon >0$. By the gluing orbit property, there exists $N=N(\vep)\in \mathbb{N}$ such that, for every $x, y \in X$ and any $n_1,n_2 \in \mathbb{N}$, there is a positive integer $t_1 \leqslant N(\vep)$ and $z \in X$ satisfying
\begin{eqnarray*}
d\big(T^j(z),\,T^j(x)\big) &\leqslant& \varepsilon \quad \quad \forall \, 0\leqslant j \leqslant n_1 \\
d\big(T^{k+n_1 + t_1}(z),\, T^k(y)\big) & \leqslant& \varepsilon \quad \quad \forall \,0\leqslant k \leqslant n_2.
\end{eqnarray*}
For each $\ell \in \mathbb{N}$, consider 
$$x_{\ell} \,=\, T^{-\ell N}(p), \quad y_{\ell} \,=\, q, \quad  n_{1,\ell} \,=\, \ell N, \quad n_{2,\ell} \,=\, \ell N.$$
Then there exist a positive integer $r_{\ell} \leqslant N(\vep)$ and $z_{\ell} \in X$ such that
\begin{eqnarray*}
d\big(T^j(z_{\ell}),\,T^{j-\ell N} (p)\big) &=& d\big(T^j(z_{\ell}),\,T^j(x_{\ell})\big) \, \leqslant \, \varepsilon \quad \quad \forall\, 0\leqslant j \leqslant n_{1,\ell} \, = \,\ell N\\
d\big(T^{k + \ell N  + r_{\ell}}(z_{\ell}), \, T^k(q)\big) &=& d\big(T^{k+n_{1,\ell}+r_{\ell}}(z_{\ell}),\, T^k(y_{\ell})\big)
\, \leqslant \, \varepsilon \quad \quad \forall\, 0\leqslant k\leqslant n_{2,\ell}\,=\,\ell N.
\end{eqnarray*}
Consider $m = j-\ell N$. Thus,
\begin{eqnarray*}
d\big(T^{m}(T^{\ell N}(z_{\ell})),\,T^{m} (p)\big) &=& d\big(T^{m+ \ell N}(z_{\ell}),\,T^{m} (p)\big)\,  \leqslant \, \varepsilon \quad \quad \forall \,-\ell N \leqslant m \leqslant 0 \\
d\big(T^{k + \ell N  + r_{\ell} -\ell N }( T^{\ell N }(z_{\ell})),\, T^k(q)\big) &=& d\big(T^{k + \ell N  + r_{\ell}}(z_{\ell}),\, T^k(q)\big) \,  \leqslant\, \varepsilon \quad \quad \forall\, 0\leqslant k\leqslant \ell N.
\end{eqnarray*}
Consequently,
\begin{eqnarray*}
d\big(T^{m}( T^{\ell N}(z_{\ell})),\,T^{m} (p) \big) &\leqslant & \varepsilon \quad \quad \forall \, -\ell N \leqslant m \leqslant 0\\
d\big(T^{k  + r_{\ell} }( T^{\ell N }(z_{\ell})),\, T^k(q)\big) &  \leqslant & \varepsilon \quad \quad 0\leqslant k\leqslant \ell N.
\end{eqnarray*}
Setting $w_{\ell} = T^{\ell N}(z_{\ell})$, we obtain 
\begin{eqnarray*}
d\big(T^{m}( w_{\ell}),\,T^{m} (p)\big) &\leqslant & \varepsilon \quad \quad -\ell N \leqslant m \leqslant 0 \\
d\big(T^{k  + r_{\ell} }( w_{\ell} ),\, T^k(q)\big) &  \leqslant &\varepsilon \quad \quad \forall\, 0\leqslant k\leqslant \ell N.
\end{eqnarray*}
Since  $r_{\ell} \leqslant N(\vep)$, we conclude that $T$ has the barycenter property in $A$.
\end{proof}

By Lemma~\ref{lemma:gluing-bary}, $f$ has the barycenter property in $A$ for every $\emptyset \neq A \subseteq M$. Moreover, as $f$ is $\Omega-$stable, one has $\emptyset \neq \per(f) = \per_h(f)$. Therefore, $f$ has the barycenter property in $\per_h(f)$, which, in this context, is equivalent to say that $f$ has the weak $su-$intersecting property in $\per(f)$. \qed

\begin{remark}
In \cite[Proposition~5.7]{CCS24}, it was proved that, assuming the shadowing property, then the topological transitivity is equivalent to the denseness of the ergodic measures in the space of invariant probability measures. 
\end{remark}

\begin{remark}
The proof of Theorem~\ref{theo-aux-complet} also ensures, without the assumption of $\Omega-$stability, that
\begin{eqnarray}\label{eq001}
\text{two-sided limit shadowing} \quad &\Rightarrow& \quad \text{$su-$intersecting property in $M$.} 
\end{eqnarray}
Moreover (see \cite[Lemma~2.1]{SSY10}),
$$\text{specification property}  \quad \Rightarrow \quad \text{weak $su-$intersecting property in $\per_{h}(f)$}.$$ 
\end{remark}

%%%%%%%%%%%%%%%%%%%%%%%%%%%%%%%%%%%%%%%%%%%%%%%%%%%%%%%%%%%%%%%%%%%%%%%%%%%%%%%%%%%%%%%%%%%%%

\section{Applications}\label{se:appl}

Let $M$ be a smooth closed Riemannian manifold and $\Diff(M)$ stand for the set of all $C^1$ diffeomorphisms of $M$, endowed with the $C^{1}$ topology. On the next pages we discuss a few consequences of our main results. They comprise the description of both the $C^1$ interior and a $C^1$ generic subset of the space of diffeomorphisms which satisfy the barycenter property in the set of periodic points; and the proof that, $C^1$ generically in the space of volume-preserving diffeomorphisms, to satisfy the barycenter property in the set of hyperbolic periodic points is equivalent to be a transitive Anosov diffeomorphism.

\subsection{Application 1}

The following is a consequence of Lemma \ref{lemma-fundamental} together with the $C^1$ generic properties within $\Diff (M)$ established by C. Bonatti and S. Crovisier in \cite{BC04} and F. Abdenur and S. Crovisier in \cite{AbCr12}.

\begin{corollary}\label{corollary:barycenter-generic-mixing}
There is a $C^1$ Baire generic set $\mathcal{R} \subset \Diff(M)$ such that, if $f\in \mathcal{R}$ and there exists $A \subseteq M$ such that  $f(A)=A$, $\Omega(f) = \overline{A}$ and $f$ has the barycenter property in $A$, then $f$ is topologically mixing.
\end{corollary}
\begin{proof}
Recall that, from \cite[Corollaire~1.2]{BC04}, there exists a Baire generic set $\mathcal{R}_{1} \subset \Diff(M)$ such that, if $f \in \mathcal{R}_{1}$, then $\Omega(f) = CR(f)$. Moreover,  as we are assuming that $M$ is connected, by \cite[Corollaire~1.3]{BC04} there exists a Baire generic set $\mathcal{R}_{2} \subset \Diff(M)$ such that, if $f \in \mathcal{R}_{2}$ and $M = CR(f)$, then $f$ is topologically transitive. From \cite[Theorem~2]{AbCr12}, there exists a Baire generic set $\mathcal{R}_{3} \subset \Diff(M)$ such that, if $f \in  \mathcal{R}_{3}$ and $f$ is topogically transitive, then $f$ is topogically mixing. Take  $f \in \mathcal{R} = \mathcal{R}_{1} \cap \mathcal{R}_{2} \cap \mathcal{R}_{3}$ and suppose that, for some $A \subseteq M$, we have $f(A)=A$, $\Omega(f)=\overline{A}$ and $f$ has the barycenter property in $A$. Therefore, by Lemma~\ref{lemma-fundamental}, $f$ is chain-transitive; and so $CR(f) = M$. Hence,  $f$ is topologically transitive. Since $f \in  \mathcal{R}_{3}$, then $f$ is topologically mixing.
\end{proof}

\subsection{Application 2}

Consider a $C^1$ diffeomorphism $f \colon M \to M$.  The index of a hyperbolic periodic point $p$ of $f$ is the dimension of the stable manifold at $p$; denote it by $\ind (p)$. Two hyperbolic periodic points $p$ and $q$ of $f$ are said to be homoclinically related if $W^s_f(\mathcal{O}_f(p)) \pitchfork  W^u_f(\mathcal{O}_f(q)) \neq \emptyset$ and $W^u_f(\mathcal{O}_f(p)) \pitchfork  W^s_f(\mathcal{O}_f(q))\neq\emptyset$, where the symbol $\pitchfork $ stands for a transversal intersection. We note that two hyperbolic periodic points which are homoclinically related have the same index (cf. \cite[Remark~1.4]{DG12}).

%%%%%%%%%%%%%%%%%%%%%%%%%%%%%%%%%%%%%%%%%%%%%%%%%%%%%%%%%%%%%

Since an Anosov diffeomorphism $f \colon M \to M$ is structurally stable (so the intersections of every stable manifold with every unstable manifold are transversal) and, by Theorem~\ref{theo-aux-complet}, a topologically transitive Anosov diffeomorphism has the weak $su-$intersecting property in $\per(f)$, we conclude that, under the transitivity assumption, every two periodic points of $f$ are homoclinically related. Consequently, by Theorem~\ref{theo-aux-complet} we conclude that:

\begin{corollary}\label{cor-index-constant}
If $f \colon M \to M$ is a topologically transitive Anosov diffeomorphism of a smooth closed Riemannian manifold $M$, then $\ind (p) = \ind (q)$ for every $p, q \in \per(f)$.
\end{corollary}

We will improve the previous corollary. A diffeomorphism $f$ is called Kupka-Smale if all the periodic points of $f$ are hyperbolic and for every pair of periodic points $p$ and $q$ of $f$ the corresponding manifolds $W^s_f(\mathcal{O}_f(p))$ and $W^u_f(\mathcal{O}_f(q))$ are transversal. Let
$$\mathcal{KS} \,=\, \big\{\text{$C^1$ Kupka-Smale diffeomorphisms of $M$}\big\}.$$
Adapting the proof of \cite[Corollary~3.2]{C15} we obtain:

\begin{corollary}\label{HMP-KS-lemma}
If $f\in \mathcal{KS}$ and $f$ has the weak $su-$intersecting property in $\per(f)$, then all the periodic points have the same index.
\end{corollary}

\begin{proof}
Let $p,q$ be two periodic points of $f$. As $f$ has the weak $su-$intersecting property in $\per(f)$, the sets $W^s_f(\mathcal{O}_f(p))\cap W^u_f(\mathcal{O}_f(q))$ and $W^u_f(\mathcal{O}_f(p))\cap W^s_f(\mathcal{O}_f(q))$ are nonempty. Taking into account that $f\in\mathcal{KS}$, we know that these intersections are transversal. Therefore, $p$ and $q$ are homoclinically related, and so they have the same index.
\end{proof}

A map $f\in\Diff (M)$ is said to be a \emph{star diffeomorphism} if there exists a $C^1$ neighborhood $\mathcal{U}$ of $f$ such that every periodic point of any $g\in\mathcal{U}$ is hyperbolic. S. Hayashi proved in \cite{H92} that being a star diffeomorphism is equivalent to being $\Omega-$stable. 

To simplify the notation, we will abbreviate the expression
\begin{center}
   ``$C^{1}$ interior of a set $S~\subseteq~\Diff (M)$"
\end{center} 
into $\intt_{C^{1}} (S)$. 

Define 
$$\mathcal{PCI} \,=\, \big\{ f \in \Diff (M)\colon \, \forall \, p, q \in \per_h(f),\, \ind (p)=\ind (q)\big\}.$$
From Corollary \ref{cor-index-constant}, we know that any   topologically transitive Anosov diffeomorphism belongs to $\intt_{C^{1}} (\mathcal{PCI})$. 
More generally: 

\begin{proposition}\label{pci-interior-omega-stable}
If $f \in \intt_{C^{1}} (\mathcal{PCI})$, then $f$ is $\Omega-$stable.
\end{proposition}

\begin{proof}
Consider $f\in \intt_{C^{1}}(\mathcal{PCI})$ and let $\mathcal{U}$ be a $C^1$ neighborhood of $f$ contained in $\mathcal{PCI}$. We claim that $f$ is a star diffeomorphism. Suppose otherwise. Then there are  $g\in\mathcal{U}$ and a non-hyperbolic periodic point $p$ of $g$. 

\begin{lemma}\label{index-function}\cite[Lemma~2.4]{SSY10}
Let $G \in \Diff (M)$ and $P$ be a non-hyperbolic periodic point of $G$. Then, for every $C^1$ neighborhood $\mathcal{V}$ of $G$ there are $H\in \mathcal{V}$ and $P_1, P_2 \in \per_h(H)$ such that $\ind(P_1)\neq \ind(P_2)$.
\end{lemma}

By Lemma~\ref{index-function}, there exist $h\in\mathcal{U}$ and two distinct hyperbolic periodic points $p_1, p_2$ of $h$ with different indices. Yet, this contradicts the fact that $ h \in \mathcal{PCI}$. Therefore, $f$ is a star diffeomorphism, and so $\Omega-$stable.
\end{proof}

\begin{corollary}\label{cor-pci-chaintransitive-anosov}
A diffeomorphism $f \in \intt_{C^{1}} (\mathcal{PCI})$ is chain-transitive if and only if it is a topologically transitive Anosov diffeomorphism.
\end{corollary}

\begin{proof}
Suppose that $f$ is chain-transitive and $f$ is in $C^{1}$ interior of $\mathcal{PCI}$. By Proposition~\ref{pci-interior-omega-stable}, the diffeomorphism $f$ is $\Omega-$stable. Since $f$ is also chain-transitive, by Theorem~\ref{theo-aux-complet} $f$ is a topologically transitive Anosov diffeomorphism. The converse assertion is clear, due to the openess of the set of topologically transitive Anosov diffeomorphisms and  Corollary~\ref{cor-index-constant}.
\end{proof} 
 
The next proposition complements the previous information on a $C^1$  generic subset of $\Diff(M)$.

\begin{proposition}\label{prop-generic-pci}
There is a $C^1$ generic set $\mathcal{R} \subset \Diff(M)$ such that every $f\in \mathcal{R}\, \cap\, \mathcal{PCI}$ is 
$\Omega-$stable.
\end{proposition}

\begin{proof}

We will adapt the proof of \cite[Theorem~C]{C15}. We start by  recalling that the existence of periodic points with different indices is a robust condition in the following sense:

\begin{lemma}\cite[Lemma~3.3]{C15}\label{b.carvalho-lemma}
There exists an $C^1$ open and dense set $\mathcal{I}$ of $\Diff(M)$ such that every $f \in \mathcal{I}$ has this property: if for any $C^1$ neighborhood $\mathcal{U}$ of $f$ some $g\in\mathcal{U}$ has two hyperbolic periodic points with different indices, then $f$ also has two hyperbolic periodic points with different indices.
\end{lemma}

Consider $f \in \mathcal{PCI} \,\cap\, \mathcal{KS}\, \cap\, \mathcal{I}$. We claim that $f$ is a star diffeomorphism. Suppose, otherwise, that for every $C^1$ neighborhood $\mathcal{U}$ of $f$ there exists some $g\in\mathcal{U}$ which  has a non-hyperbolic periodic point. We perturb $g$ using Lemma~\ref{index-function} to obtain $h\in\mathcal{U}$ such that $h$ has two distinct hyperbolic periodic points with different indices.  Since $f \in\mathcal{I}$, by Lemma~\ref{b.carvalho-lemma} we conclude that $f$ has two distinct hyperbolic periodic points with different indices as well. Yet, as $f$ belongs in $\mathcal{PCI}$, this is a contradiction. Thus, $f$ is $\Omega-$stable. So, we take 
\begin{equation}\label{def:R}
\mathcal{R} \,=\,  \mathcal{KS} \cap \mathcal{I}.
\end{equation}
\end{proof}

\begin{corollary}\label{generic-pci-chain-transitive}
If $f \in \mathcal{R}\,\cap\, \mathcal{PCI}$ and is chain-transitive, then $f$ is a topologically transitive Anosov diffeomorphism.
\end{corollary}

\begin{proof}
Take a chain-transitive $f \in \mathcal{R}\,\cap\, \mathcal{PCI}$. By Proposition \ref{prop-generic-pci}, $f$ is $\Omega-$stable. Then, by Theorem~\ref{theo-aux-complet}, $f$ is a topologically transitive Anosov diffeomorphism. 
\end{proof}

\subsection{Application 3}

Let $M$ be a smooth closed Riemannian manifold with dimension $\geqslant 2$ and $f\colon M \to M$ be a $C^1$ diffeomorphism. Recall that if $f$ has the weak $su-$intersecting property in $\per(f)$, then $f$ has the barycenter property in $\per(f)$; furthermore, if $\per(f) = \per_h(f)$, then $f$ has the weak $su-$intersecting property in $\per(f)$ if and only if $f$ has the barycenter property in $\per(f)$. Since in $\mathcal{KS}$ we have $\per(f) = \per_h(f)$, then generically $f$ has the weak $su-$intersecting property in $\per(f)$ if and only if $f$ has the barycenter property in $\per(f)$.

Consider the following subsets of $\Diff (M)$:
\begin{eqnarray*}
\mathcal{CT} &=& \big\{f \in \Diff (M)\colon \, \text{$f$ is chain-transitive}\big\} \\
\mathcal{GL} &=& \big\{f \in \Diff (M)\colon \, \text{$f$ has the gluing orbit property}\big\} \\
\mathcal{SPE} &=& \big\{f \in \Diff (M)\colon \, \text{$f$ has the specification property}\big\} \\
\mathcal{TLS} &=& \big\{f \in \Diff (M)\colon \,\text{$f$ has the two-sided limit shadowing property}\big\} \\
\mathcal{IP} &=& \big\{f \in \Diff (M) \colon \,\text{$f$ has the $su-$intersecting property in $M$}\big\}\\
\mathcal{IP}_{h} &=& \big\{f \in \Diff (M)\colon \, \text{$f$ has the $su-$intersecting property in $\per_{h}(f)$}  \big\}\\
\mathcal{WIP}_{h} &=& \big\{f \in \Diff (M)\colon \, \text{$f$ has the weak $su-$intersecting property in $\per_{h}(f)$}\big\}.
\end{eqnarray*}

J. Palis and S. Newhouse proved in \cite{Palis-New} that, if $f\colon M \to M$ is a $C^1$ diffeomorphism of a closed Riemannian two-dimensional manifold whose non-wandering set is hyperbolic, then $\Omega(f) = \overline{\per(f)}$ and $f$ is $C^1$ approximated by an $\Omega-$stable diffeomorphism $g$ such that $\Omega(g) = \Omega(f)$. Therefore, if such an $f$ belongs in the $C^1$ interior of $\mathcal{WIP}_{h}$, then  Theorem~\ref{theo-aux-complet} yields that $f$ is a topologically transitive Anosov diffeomorphism. In the next subsection we provide a broader framework to this conclusion.

\subsubsection{\emph{\textbf{The $C^1$ interior of $\mathcal{WIP}_{h}$}}}

In \cite{SSY10}, K. Sakai, N. Sumi and K. Yamamoto showed that  $\intt_{C^{1}} (\mathcal{SPE})$ is equal to the space of  topologically transitive Anosov diffeomorphisms. In \cite{C15}, B. Carvalho proved that $\intt_{C^{1}} (\mathcal{TLS})$  is also equal to the set of topologically transitive Anosov diffeomorphisms. In \cite{L16-2}, M. Lee proved that the $C^{1}$ interior of the diffeomorphisms satisfying the barycenter property in $\per(f)$ is contained in the set of $\Omega-$stable diffeomorphisms. We extend these results as follows.

\begin{proposition}\label{theo-c1-interior}
The set $\intt_{C^{1}}(\mathcal{WIP}_{h})$ coincides with the space of topologically transitive Anosov diffeomorphisms.  Moreover, $\intt_{C^{1}}(\mathcal{WIP}_{h})  =  \intt_{C^{1}}(\mathcal{PCI}) \cap  \mathcal{CT}$.
\end{proposition}

\begin{proof}

We start by establishing the following inclusion.

\begin{lemma}\label{interior-wip-interior-pci}
$\intt_{C^{1}}(\mathcal{WIP}_{h}) \subseteq \intt_{C^{1}}(\mathcal{PCI})$.
\end{lemma}

\begin{proof}
We will adapt the proof of \cite[Theorem~A]{C15}. Consider $f\in \intt_{C^{1}}(\mathcal{WIP}_{h})$ and let $\mathcal{U}$ be a $C^1$ open neighborhood of $f$ contained in $\mathcal{WIP}_{h}$.  Suppose that $f \notin \intt_{C^{1}}(\mathcal{PCI})$. Thus, for any $C^1$ open neighborhood $\mathcal{V}$ of $f$, one has $\mathcal{V} \cap \big(\Diff(M) \setminus \mathcal{PCI}\big) \neq \emptyset$. Take $g \in \mathcal{U}  \cap \big(\Diff(M) \setminus \mathcal{PCI}\big)$. Since $\per_{h}(g) \neq  \emptyset$ because $g \in \mathcal{WIP}_{h}$, and $g \notin  \mathcal{PCI}$, there are two hyperbolic periodic points $p_1, p_2$ of $g$ with different indices. Taking into account that  these points are hyperbolic and $\mathcal{KS}$ is $C^1$ generic in $\Diff (M)$, we may perturb $g$ to obtain $\tilde{g} \in \mathcal{U} \cap \mathcal{KS}$ such that $\tilde{g}$ has two  hyperbolic periodic points $\tilde{p}_1, \tilde{p}_2$ with different indices. But this contradicts Corollary~\ref{HMP-KS-lemma}. So $f \in \intt_{C^{1}}(\mathcal{PCI})$.
\end{proof}

We proceed by showing that $\intt_{C^{1}}(\mathcal{WIP}_{h})  =  \intt_{C^{1}}(\mathcal{PCI}) \cap  \mathcal{CT}$.

\noindent $(\subseteq)$ Take $f \in \intt_{C^{1}}(\mathcal{WIP}_{h})$. By Lemma \ref{interior-wip-interior-pci}, $f \in \intt_{C^{1}}(\mathcal{PCI})$, and,  by Proposition \ref{pci-interior-omega-stable}, $f$ is $\Omega-$stable. Since $f$ also has the weak $su-$intersecting property in $\per(f)$, then $f$ is chain-transitive by Theorem~\ref{theo-aux-complet}. So, $f \in \intt_{C^{1}}(\mathcal{PCI}) \cap  \mathcal{CT}$.

\noindent $(\supseteq)$ Suppose that $f \in  \intt_{C^{1}}(\mathcal{PCI}) \cap  \mathcal{CT}$. By Corollary \ref{cor-pci-chaintransitive-anosov}, $f$ is a topologically transitive Anosov diffeomorphism. As the space of topologically transitive Anosov diffeomorphisms is open, there exists a $C^1$ open neighborhood $\mathcal{U}$ of $f$ such that any $g \in \mathcal{U}$ is a topologically transitive Anosov diffeomorphism. By Theorem~\ref{theo-aux-complet}, every $g \in \mathcal{U}$ has the weak $su-$intersecting property in $\per(g)$. So, $f \in \intt_{C^{1}}(\mathcal{WIP}_{h})$.

We are ready to finalize the proof of Proposition~\ref{theo-c1-interior} since, by Corollary~\ref{cor-pci-chaintransitive-anosov}, a diffeomorphism $f$ belongs in $\intt_{C^{1}}(\mathcal{WIP}_{h}) = \intt_{C^{1}}(\mathcal{PCI}) \cap  \mathcal{CT}$ if and only if it is a topologically transitive Anosov diffeomorphism. 
\end{proof}

\begin{corollary}\label{cor-c1-interior-HMP} 
$$\intt_{C^{1}} (\mathcal{GL})= \intt_{C^{1}} (\mathcal{TLS}) = \intt_{C^{1}} (\mathcal{SPE})= \intt_{C^{1}} (\mathcal{IP}) = \intt_{C^{1}} (\mathcal{IP}_{h}) =  \intt_{C^{1}} (\mathcal{WIP}_{h}).$$
\end{corollary}

\begin{proof}
Firstly, by \eqref{eq001}, \cite[Lemma~2.1]{SSY10} and Lemma \ref{lemma:gluing-bary}, one has 
\begin{eqnarray}\label{eq:in}
\intt_{C^{1}} (\mathcal{TLS}) \,\subseteq\, \intt_{C^{1}} (\mathcal{IP}) &\subseteq & \intt_{C^{1}} (\mathcal{IP}_{h}) \,\subseteq\, \intt_{C^{1}} (\mathcal{WIP}_{h}) \nonumber \\
\intt_{C^{1}} (\mathcal{SPE}) &\subseteq &  \intt_{C^{1}} (\mathcal{WIP}_{h}) \nonumber\\
\intt_{C^{1}} (\mathcal{GL}) & \subseteq &  \intt_{C^{1}} (\mathcal{WIP}_{h}).
\end{eqnarray}
Now, take $f \in  \intt_{C^{1}} (\mathcal{WIP}_{h})$ and  $\mathcal{U}$ a $C^1$ open neighborhood of  $f$ contained in  $ \mathcal{WIP}_{h}$. By Proposition~\ref{theo-c1-interior}, every $g \in \mathcal{U}$ is a topologically transitive Anosov diffeomorphism. In particular, every $g \in \mathcal{U}$ is $\Omega-$stable and has the weak $su-$intersecting property in $\per(g)$. Thus, by Theorem~\ref{theo-aux-complet}, every $g \in \mathcal{U}$ has the two-sided limit shadowing property, the specification property and the gluing orbit property. Hence, 
$$f \in \intt_{C^{1}} (\mathcal{TLS}) \,\cap \, \intt_{C^{1}} (\mathcal{SPE}) \,\cap \, \intt_{C^{1}} (\mathcal{GL}).$$
Thus, 
\begin{equation}\label{eq:out}
\intt_{C^{1}} (\mathcal{WIP}_{h}) \,\subseteq \,\intt_{C^{1}} (\mathcal{TLS}) \, \cap \, \intt_{C^{1}} (\mathcal{SPE}) \, \cap \,  \intt_{C^{1}} (\mathcal{GL}).
\end{equation}
From the inclusions in \eqref{eq:in} and \eqref{eq:out}, we conclude that all those sets are equal.
\end{proof}

\subsubsection{\emph{\textbf{Generic elements of $\mathcal{WIP}_{h}$}}}\label{sse:generic}

In \cite[Theorem~C]{C15}, B. Carvalho showed that $C^1$ generically if a diffeomorphism $f \colon M \to M$ satisfies the two-sided limit shadowing property then $f$ is a topologically transitive Anosov diffeomorphism. In \cite[Corollary~1.2]{LT15}, K. Lee and K. Tajbakhsh established that $C^1$ generically if $f$ has the specification property then $f$ is a topologically transitive Anosov diffeomorphism. In \cite{L16-2}, M. Lee proved  that $C^1$ generically if $f$ has the barycenter property in $\per(f)$ then it is $\Omega-$stable. We improve these results as follows.

\begin{proposition}\label{c1-generic}
Let $\mathcal{R} =  \mathcal{KS} \cap \mathcal{I}$ be the set defined in \eqref{def:R}. If $f\in \mathcal{R}\, \cap\, \mathcal{WIP}_{h}$, then $f$ is a topologically transitive Anosov diffeomorphism. 
\end{proposition}

\begin{proof} We start by establishing that, generically, $\mathcal{WIP}_{h} = \mathcal{PCI} \cap \mathcal{CT}$.

\begin{lemma}\label{le:aux} $\mathcal{R}\, \cap\, \mathcal{WIP}_{h} = \mathcal{R}\, \cap\, \mathcal{PCI} \,\cap\, \mathcal{CT}$.
\end{lemma}

\begin{proof}
\noindent $(\subseteq)$ Take $f \in \mathcal{R}\, \cap\, \mathcal{WIP}_{h}$, that is, $f\in \mathcal{KS}$ and $f$ has the weak $su-$intersecting property in $\per(f)$. By Corollary \ref{HMP-KS-lemma}, all periodic points of $f$ are hyperbolic  and have the same index. Thus, $f \in \mathcal{PCI}$. By Proposition~\ref{prop-generic-pci}, $f$ is $\Omega-$stable. Since, in addition, $f$ has the weak $su-$intersecting property in $\per(f)$, from Theorem~\ref{theo-aux-complet} we deduce that $f$ is chain-transitive. Hence, $f \in \mathcal{R}\, \cap\, \mathcal{PCI} \, \cap \, \mathcal{CT}$.

\noindent $(\supseteq)$ Suppose that $f \in \mathcal{R}\, \cap\, \mathcal{PCI} \cap \mathcal{CT}$. By Proposition \ref{prop-generic-pci}, $f$ is $\Omega-$ stable. Thus, $f$ is  chain-transitive. Therefore, by Theorem~\ref{theo-aux-complet}, $f$ has the weak $su-$intersecting property in $\per(f)$.
\end{proof}

Let us end the proof of Proposition~\ref{c1-generic}. Suppose that $f\in \mathcal{R}\, \cap\, \mathcal{WIP}_{h}$. By  Lemma~\ref{le:aux}, $f \in \mathcal{R}\, \cap\, \mathcal{PCI}\, \cap\, \mathcal{CT}$. Consequently, by Proposition~\ref{prop-generic-pci}, $f$ is $\Omega-$ stable. Since $f$ is also chain-transitive, Theorem~\ref{theo-aux-complet} ensures that $f$ is a topologically transitive Anosov diffeomorphism. 
\end{proof}

\begin{corollary}\label{cor-c1-generic}
For every $f \in \mathcal{R}$, the following assertions are equivalent:
\begin{itemize}
\item[$(1)$]  $f \in \mathcal{TLS}$.
\item[$(2)$]  $f \in \mathcal{IP}$.
\item[$(3)$]  $f \in \mathcal{IP}_h$.
\item[$(4)$]  $f \in \mathcal{WIP}_h$.
\item[$(5)$]  $f \in \mathcal{PCI}\,\cap\, \mathcal{CT}$.
\item[$(6)$]  $f$ is a topologically transitive Anosov diffeomorphism.
\item[$(7)$]  $f \in \mathcal{SPE}$.
\item[$(8)$]  $f \in \mathcal{CT}$ and has the shadowing property.
\item[$(9)$]  $f \in \mathcal{GL}$.
\end{itemize}
\end{corollary}

\begin{proof}

\noindent $(1) \Rightarrow (2) \Rightarrow (3) \Rightarrow (4)$  See \eqref{eq001}.

\noindent $(4) \Rightarrow (5)$ Consider $f \in \mathcal{R}\, \cap\, \mathcal{WIP}_{h}$. By Proposition~\ref{c1-generic}, one has $ \mathcal{R}\, \cap\, \mathcal{WIP}_{h} =  \mathcal{R}\, \cap\, \mathcal{PCI} \cap \mathcal{CT}$. So, $f \in \mathcal{R}\, \cap\, \mathcal{PCI} \cap \mathcal{CT}$.

\noindent $(5) \Rightarrow (6)$ This is Corollary~\ref{generic-pci-chain-transitive}.

\noindent  $(6) \Rightarrow (1)$ Let $f$ be a topologically transitive Anosov diffeomorphism. So $f$ is an expansive topologically transitive diffeomorphism that has the shadowing property. Since $M$ is connected, $f$ is topologically mixing. By \cite[Theorem~45]{KLO16}, $f$ has the specification property. Thus, according to Proposition~\ref{theo-Homeo-hyperbolicmeet}, $f$ has the two-sided limit shadowing property.

\noindent $(6) \Rightarrow (7)$ This was established in the implication $(6) \Rightarrow (7)$ of Theorem~\ref{theo-aux-complet}.

\noindent  $(7) \Rightarrow (4)$ Assume that $f$ has the specification property. By \cite[Lemma~2.1]{SSY10}, we know that $f$ has the weak $su-$intersecting property in $\per_{h}(f)$.

\noindent  $(6) \Rightarrow (8)$ This is clear.

\noindent  $(8) \Rightarrow (7)$ Suppose that $f$ is a chain-transitive diffeomorphism which satisfies the shadowing property. Then $f$ is topologically transitive, since chain-transitivity together with the shadowing property imply topological transitivity. Therefore, as $M$ is connected and $f$ has the shadowing property, we conclude that $f$ is topologically mixing. By \cite[Theorem~45]{KLO16}, $f$ has the specification property.

\noindent  $(6) \Rightarrow (9)$ Suppose that $f$ is a topologically transitive Anosov diffeomorphism. So, besides being topologically transitive, $f$ has the  shadowing property. Since $M$ is connected, $f$ has the gluing orbit property by Lemma~\ref{lemma:sh+tr+connct-gluing}.

\noindent  $(9) \Rightarrow (4)$ Suppose that $f$ has the gluing orbit property. By Lemma \ref{lemma:gluing-bary}, $f$ has the barycenter property in $A$ for every nonempty $A$ of $M$. Since $\mathcal{R} \subseteq \mathcal{KS}$, we have $\emptyset \neq \per(f) = \per_h(f)$. Thus, $f$ has the barycenter property in $\per_h(f)$, and so it has the weak $su-$intersecting property in $\per(f)$.
\end{proof}

%%%%%%%%%%%%%%%%%%%%%%%%%%%%%%%%%%%%%%%%%%%%%%%%%%%%%%%%%%%%%%%%%%%%%%%%%%%%%%%%%%%%%%%%%%%%%

\subsection{Application 4}

The proof of Theorem~\ref{theo-aux-complet} ensures, without the assumption of $\Omega-$stability, that
\begin{eqnarray*} 
\text{specification property} &\Rightarrow &  \text{almost specification property} \\
&\Rightarrow & \text{asymptotic average shadowing property} \\
&\Rightarrow & \text{average shadowing property} \\
&\Rightarrow &  \text{chain-transitivity}
\end{eqnarray*}
and it is clear that one always has
$$\text{topological transitivity} \quad \Rightarrow \quad \text{chain-transitivity}.$$ 
As regards these properties, in \cite{SA00}, K. Sakai showed that the $C^{1}$ interior of the set of diffeomorphisms of a compact surface  satisfying the average shadowing property is equal to the set of topologically transitive Anosov diffeomorphisms. On the other hand, in \cite{M82}, R. Ma\~n\'e proved that every $C^{1}-$robustly transitive diffeomorphism of a compact surface is a topologically transitive Anosov diffeomorphism. 
As a consequence of Theorem~\ref{theo-aux-complet} together with Corollary~\ref{cor-index-constant}, we provide a unique proof for both these   results.

\begin{proposition}\label{theo-surface}
The $C^{1}$ interior of the set of chain-transitive diffeomorphisms of a compact surface is made of topologically transitive Anosov diffeomorphisms.
\end{proposition}

\begin{proof} Our argument is inspired by the proof of \cite[Theorem~2]{L18}. Let $M$ be a compact surface and $f\colon M \to M$ be a $C^1$ diffeomorphism in the interior of the set of chain-transitive diffeomorphisms of $M$. We claim that $f$ is a star diffeomorphism. 

Suppose, otherwise, that $f$ is not star. Let $\mathcal{U}$ be a $C^1$ neighborhood of $f$ contained in the set of chain-transitive diffeomorphisms of $M$. Then there exist $g\in\mathcal{U}$ and a non-hyperbolic periodic point $p$ of $g$. In addition, by Lemma~\ref{index-function}, there exist $h\in\mathcal{U}$ and two distinct hyperbolic periodic points $p_1,p_2$ of $h$ with different indices. 

\begin{lemma}\label{cor-mane}\cite[Corollary~II]{M82}
If $M$ is a compact surface, then there exists a residual subset $\mathcal{Z} \subseteq \diff^{1}(M)$ such that every $f \in \mathcal{Z}$ satisfies one of the following properties:
\begin{itemize}
\item[$(a)$] $f$ has infinitely many sinks.
\item[$(b)$] $f$ has infinitely many sources.
\item[$(c)$] $f$ is $\Omega-$stable.
\end{itemize}
\end{lemma}

According to Lemma~\ref{cor-mane} and Lemma~\ref{b.carvalho-lemma}, we may perturb $h$ to obtain $\tilde{h}\in\mathcal{U} \cap \mathcal{Z} \cap \mathcal{I}$ such that $\tilde{h}$ has two distinct hyperbolic periodic points $\tilde{p}_1, \tilde{p}_2$ with different indices. Moreover, as $\tilde{h} \in \mathcal{U}$, then $\tilde{h} $ is chain-transitive.

\begin{lemma}\label{chain-lemma-sinks-sources}
If a diffeomorphism is chain-transitive then it has neither sinks nor sources.
\end{lemma}

\begin{proof}
Let $H$ be a chain-transitive diffeomorphism. Then $H^{-1}$ is also chain-transitive (cf. \cite{W16}). Moreover, $H$ has no sinks (cf.  \cite[Lemma~2.1]{L15}). Since a source of $H$ is a sink of $H^{-1}$, we deduce that $H$ has no sources as well.
\end{proof}

Let us resume the proof of Proposition~\ref{theo-surface}. As $\tilde{h} $ is chain-transitive, by Lemma~\ref{chain-lemma-sinks-sources} we know  that $\tilde{h}$ does not have sinks or sources. Therefore, as $\tilde{h} \in \mathcal{Z}$, Lemma~\ref{cor-mane} yields that $\tilde{h}$ is $\Omega-$stable. Thus, Theorem~\ref{theo-aux-complet} ensures that $\tilde{h}$ is a topologically transitive Anosov diffeomorphism. Hence, by Corollary~\ref{cor-index-constant}, for any pair of hyperbolic periodic points $p$ and $q$ of $\tilde{h}$ one has $\ind (p)=\ind (q)$. Yet,  this contradicts the fact that $\tilde{h}$ has two hyperbolic periodic points $\tilde{p}_1, \tilde{p}_2$ with different indices. Consequently, $f$ is $\Omega-$stable. Since $f$ is also chain-transitive, Theorem~\ref{theo-aux-complet} implies that $f$ is a topologically transitive Anosov diffeomorphism. 
\end{proof}

\begin{remark}
When the dimension of the manifold is equal to $3$ or $4$ it is known that a $C^{1}-$robustly transitive diffeomorphism may not be an Anosov diffeomorphism \cite{M78,S71}. On the other hand, in \cite{AR11}, A. Arbieto and R. Ribeiro showed that the $C^1$  interior of the set of flows with the asymptotic average shadowing property (or with the average shadowing property) on a $3-$dimensional smooth closed manifold is formed by topologically transitive Anosov flows.
\end{remark}

%%%%%%%%%%%%%%%%%%%%%%%%%%%%%%%%%%%%%%%%%%%%%%%%%%%%%%%%%%%%%%%%%%%%%%%%%%%%%%%%%%
 
\subsection{Application 5}

Let $M$ be a smooth closed Riemannian manifold and $\Diff_{m} (M)$ denote the set of diffeomorphisms which preserve the Lebesgue measure $m$. We endow $\Diff_{m} (M)$ with the $C^{1}$ topology. In this subsection we shall prove that a volume-preserving diffeomorphism has a constant index within the set of hyperbolic periodic points if and only if it a transitive  Anosov diffeomorphism. 

Denote by $\mathcal{F}^{1}_{m} (M)$ the set of \emph{star diffeomorphisms} in $\Diff_{m}(M)$, that is, those $f \in \Diff_{m}(M)$  which have a $C^{1}$ neighborhood $U \subseteq \Diff_{m}(M)$ such that, if $g \in U$, then any periodic point of $g$ is hyperbolic.

\begin{proposition}
The following assertions are equivalent regarding $f \in \Diff_{m}(M)$:
\begin{itemize}
\item[$(a)$] $f$ is an Anosov diffeomorphism.
\item[$(b)$] $f \in \mathcal{F}^{1}_{m} (M)$.
\item[$(c)$] There exist a $C^{1}$ neighborhood $U \subseteq \Diff_{m}(M)$ of $f$ and a positive integer $I_0$ such that, for every diffeomorphism $g \in U$ and every hyperbolic periodic $p$ of $g$, the index of $p$ with respect to $g$ is $I_0$.
\end{itemize}
\end{proposition}

\begin{proof}
$(a) \Rightarrow (b)$ This is clear.

\noindent $(b) \Rightarrow (c)$ This was established in \cite[Proposition~3.1]{ArCa13}.

\noindent $(c) \Rightarrow (a)$ Take $f \in \Diff_{m}(M)$  complying with the condition $(c)$. We claim that $f$ is an Anosov diffeomorphism. Assume otherwise. By \cite[Theorem~1.1]{ArCa13}, $f \notin \mathcal{F}^{1}_{m} (M)$, so there exists $g \in \Diff_{m}(M)$ close enough to $f$ with a non-hyperbolic periodic point $p$. Thus, according to \cite[Lemma~5.1]{ArCa13}, there exists $h \in \Diff_{m}(M)$ which is close enough to $f$ and has two periodic points $q$ and $r$ with different indices. This contradicts $(c)$.
\end{proof}

We may improve the previous proposition in the following generic sense.

\begin{proposition}\label{c1-volume-preserving -generic}
There is a Baire generic set $\mathcal{R}_{m} \subset \Diff_{m}(M)$ such that for every $f \in \mathcal{R}_{m}$ the following assertions are equivalent:
\begin{itemize}
\item[$(1)$]  $f$ has the barycenter property in $\per_{h}(f)$.
\item[$(2)$]  $f$ has the weak $su-$intersecting property in $\per_{h}(f)$.
\item[$(3)$]  For any $\, p, q \in \per_h(f)$, one has $\ind (p)=\ind (q)$.
\item[$(4)$]  $f \in \mathcal{F}^{1}_{m} (M)$.
\item[$(5)$]  $f$ is a transitive Anosov diffeomorphism.
\end{itemize}
\end{proposition}

\begin{proof}

We start by establishing an analogue of Lemma~\ref{b.carvalho-lemma} in the context of volume-preserving diffeomorphisms. 

\begin{lemma}\label{b.carvalho-lemma-volume-preserving}
There exists an $C^1$ open and dense set $\mathcal{J}$ of $\Diff_{m}(M)$ such that, if $f \in \mathcal{J}$ and in every $C^1$ neighborhood $\mathcal{U}$ of $f$ some $g\in\mathcal{U}$ has two hyperbolic periodic points with different indices, then $f$ also has two hyperbolic periodic points with different indices.
\end{lemma}

\begin{proof}
Let $\mathcal{A}$ be the set of $C^{1}$ diffeomorphisms which have two hyperbolic periodic points with different indices. The hyperbolicity of these periodic points guarantees that 
$\mathcal{A}$ is open in $\Diff(M)$. So, $\mathcal{A}_{*} = \mathcal{A} \cap \Diff_{m}(M)$ is an open set of $ \Diff_{m}(M)$. Consider the complement $\mathcal{B} = \Diff_{m}(M) \setminus \overline{\mathcal{A}_{*}}^{\Diff_{m}(M)}$. Note that $B$ is also open in $\Diff_{m}(M)$; hence $\mathcal{J} = \mathcal{A}_{*} \cup \mathcal{B}$ is open in $\Diff_{m}(M)$. We claim that $\mathcal{J}$ is also dense in $\Diff_{m}(M)$.  Indeed, if $g \in \Diff_{m}(M) \setminus \mathcal{B}$, then $g \in \overline{\mathcal{A}_{*}}^{\Diff_{m}(M)}$, and so it is approximated by diffeomorphisms of $\mathcal{A}_{*} \subseteq \mathcal{J}$. 

Let $f\in \mathcal{J}$ and suppose that there is a sequence $(f_{n})_{n\, \in\, \mathbb{N}}$ of diffeomorphisms in $\Diff_{m}(M)$ converging to $f$ in the $C^{1}$ topology and such that each $f_{n}$ has two hyperbolic periodic points with different indices. Then $f_{n} \in \mathcal{A}_{*}$ for all $n \in N$. Thus, $f \in \overline{\mathcal{A}_{*}}^{\Diff_{m}(M)}$, so $f \notin \mathcal{B}$. Since $ \mathcal{J} = \mathcal{A}_{*} \cup \mathcal{B}$ and $f \in \mathcal{J}$, we conclude that $f \in  \mathcal{A}_{*}$. That is, $f$ has two hyperbolic periodic points with different indices.
\end{proof}

Denote by $\mathcal{KS}_{m}$ the $C^1$ generic set of Kupka-Smale volume-preserving diffeomorphisms provided by Robinson's theorem (see \cite{Rob70}). The next lemma is a counterpart in the conservative setting of Corollary~\ref{HMP-KS-lemma}, and is proved by a similar argument.

\begin{lemma}\label{HMP-KS-lemma-vol-preserving}
If $f\in \mathcal{KS}_{m}$ and $f$ has the weak $su-$intersecting property in $\per(f)$, then all the periodic points have the same index.
\end{lemma}

Let us resume the proof of Proposition~\ref{c1-volume-preserving -generic}. Define $\mathcal{R}_{m} = \mathcal{KS}_{m} \cap \mathcal{J}$ and take $f\in  \mathcal{R}_{m}$.

\noindent $(1) \Rightarrow (2)$ Suppose that $f$ has the barycenter property in $\per_{h}(f)$. Since $f \in \mathcal{KS}_{m}$, one has $\per_{h}(f)= \per(f)$, and so $f$ satisfies the weak $su-$intersecting property in $\per_{h}(f)$. 

\noindent $(2) \Rightarrow (3)$ Assume that $f$ has the weak $su-$intersecting property in $\per_{h}(f)$. 

Lemma \ref{HMP-KS-lemma-vol-preserving} ensures that all the periodic points have the same index since $f \in \mathcal{KS}_{m}$.

\noindent $(3) \Rightarrow (4)$ Suppose that, for any $\, p, q \in \per_h(f)$, we have $\ind (p)=\ind (q)$. We claim that $f \in \mathcal{F}^{1}_{m} (M)$. Assume otherwise, that $f \notin \mathcal{F}^{1}_{m} (M)$, so there exists $g \in \Diff_{m}(M)$ close enough to $f$ with a non-hyperbolic periodic point $p$. Thus, according to \cite[Lemma~5.1]{ArCa13}, there exists $h \in \Diff_{m}(M)$ which is close enough to $f$ and has two periodic points with different indices. Since $f \in \mathcal{J}$, by Lemma~\ref{b.carvalho-lemma-volume-preserving}, $f$ has two hyperbolic periodic points with different indices. Yet, this contradicts the assumption on $f$. Therefore, $f \in \mathcal{F}^{1}_{m} (M)$.

\noindent $(4) \Rightarrow (5)$  This was established in \cite[Theorem~1.1]{ArCa13}.

\noindent $(5) \Rightarrow (1)$ This is clear by Theorem~\ref{maintheo-weak-hmp}(c). 
\end{proof}

%%%%%%%%%%%%%%%%%%%%%%%%%%%%%%%%%%%%%%%%%%%%%%%%%%%%%%%%%%%%%%%%%%%%%%%%%%%%%%%%%%%%%%%%%%%%%

\subsection*{Acknowledgments}
L.S. is partially supported by FAPERJ-Funda\c c\~ao
  Carlos Chagas Filho de Amparo \`a Pesquisa do Estado do Rio de
  Janeiro Projects APQ1-E-26/211.690/2021 SEI-260003/015270/2021 and
  JCNE-E-26/200.271/2023 SEI-260003/000640/2023, CAPES-Coordena\c c\~ao
  de Aperfei\c{c}oamento de Pessoal de N\'ivel Superior - Finance
  Code 001 and PROEXT-PG project Dynamic Women - Din\^amicas,
  Projeto Universal CNPq 404943/2023-3 and Projeto Universal CNPq 401737/2025-0. 
  MC is partially supported by Portuguese public funds through 
  FCT - Funda\c c\~ao para a Ci\^encia e a Tecnologia, I.P., 
  under the projects with references UIDB/00144/2020, UIDP/00144/2020 and 
  PTDC/MAT-PUR/4048/2021.

%%%%%%%%%%%%%%%%%%%%%%%%%%%%%%%%%%%%%%%%%%%%%%%%%%%%%%%%%%%%%%%%%%%%%%%%%%%%%%%%%%%%%%%%%%%%%

\def\cprime{$'$}

\end{document}